\documentclass[a4paper, 10pt, UTF8]{amsart}
\usepackage{amsthm,amsmath,amssymb,amsopn}
\usepackage{enumerate}
\usepackage{graphics}
\usepackage{float}
\usepackage[]{color}
\usepackage[left=3cm, right=3cm]{geometry}
\usepackage[colorlinks]{hyperref}
\usepackage{tikz}
\usepackage{multirow}
\usepackage{algorithm}
\usepackage{algorithmic}

\newcommand{\bm}[1]{\boldsymbol{#1}}
\newcommand{\lj}{[ \hspace{-2pt} [}
\newcommand{\rj}{] \hspace{-2pt} ]}
\newcommand{\mb}[1]{\mathbb{#1}}
\newcommand{\mc}[1]{\mathcal{#1}}
\newcommand{\mr}[1]{\mathrm{#1}}
\newcommand{\jump}[1]{\lj #1 \rj}
\newcommand{\aver}[1]{ \{#1\}  }
\newcommand{\enorm}[1]{ |\!|\!| #1 |\!|\!|}
\newcommand{\wnorm}[1]{ \| #1 \|_{*}}

\newcommand{\wt}[1]{\widetilde{#1}}
\renewcommand{\d}[1]{\mathrm d \boldsymbol{#1}}

\newcommand{\revise}[1]{#1}
\newcommand{\substitute}[2]{{#2}}

\def\MTh{\mc{T}_h}
\def\MThG{\mc{T}_h^\Gamma}
\def\MThB{\mc{T}_h^{\backslash \Gamma}}
\def\MEh{\mc{E}_h}
\def\MEhG{\mc{E}_h^\Gamma}
\def\MEhB{\mc{E}_h^{\backslash \Gamma}}
\def\un{\bm{\mr n}}

\newtheorem{assumption}{Assumption}
\newtheorem{theorem}{Theorem}
\newtheorem{lemma}{Lemma}
\newtheorem{remark}{Remark}

\begin{document}

\title[Interface Problem]{A Discontinuous Galerkin Method by Patch
  Reconstruction for Elliptic Interface Problem on Unfitted Mesh}

\author[R. Li]{Ruo Li} \address{CAPT, LMAM and School of Mathematical
  Sciences, Peking University, Beijing 100871, P.R. China}
\email{rli@math.pku.edu.cn}

\author[F.-Y. Yang]{ Fanyi Yang} \address{School of Mathematical
  Sciences, Peking University, Beijing 100871, P.R. China}
\email{yangfanyi@pku.edu.cn}

\maketitle

\begin{abstract}
  We propose a discontinuous Galerkin (DG) method to approximate the
  elliptic interface problem on unfitted mesh using a new
  approximation space. The approximation space is constructed by patch
  reconstruction with one degree of freedom per element. 
  \substitute{The optimal
  error estimates in both $L^2$ norm and DG energy norm are obtained,
  without the typical constraints for DG method on how the interface
  intersects to the elements in the mesh.}{\revise{
  The optimal error estimates in both $L^2$ norm and DG energy norm
  are obtained, without restrictions on how the interface intersects
  the elements in the mesh.  The stability near the interface is
  ensured by the patch reconstruction and no special numerical flux is
  required. }}
  The convergence order by
  numerical results in both 2D and 3D agrees with the error estimates
  perfectly. More than enjoying the advantages of DG method, the new
  method may achieve even better efficiency in number of degree of
  freedom than the conforming finite element method as illustrated by
  our numerical examples.

  \noindent \textbf{keywords}: Elliptic interface problems, Patch
  reconstructed, Discontinuous Galerkin method, Unfitted mesh.

\end{abstract}


\section{Introduction}
\label{sec:intro}

In the last decades, numerical methods for the elliptic interface
problem have attracted pervasive attention since the pioneering work
of Peskin \cite{peskin1977blood}, for example, the immersed interface
method by LeVeque and Li \cite{Leveque1994immersed, Li2006immersed},
Mayo's method on irregular regions \cite{Mayo1985fast}, the method in
\cite{Yu2007matched} with second-order accuracy in the $L^\infty$
norm. In the finite difference fold, we also refer to
\cite{Liu2000boundary, Hou2005numerical, Hou2010numerical,
  Fedkiw1999ghost, Chen2008piecewise, Oevermann2006Cartesian} for some
other interesting methods. Meanwhile, finite element (FE) method is
also popular for solving the interface problem. Based on the
geometrical relationship between the grid and the interface, FE
methods could be classified into two categories: interface-fitted
method and interface-unfitted method. The body-fitted grid enforces
the mesh to align with the interface to render a high-order accurate
approximation \cite{Chen1998interface, Barrett1987fitted}. However,
generating a fitted mesh with satisfied quality is sometimes a
nontrivial and time-consuming task \cite{Wu2012unfitted,
Xiao2020high}. Therefore, there are some techniques for FE methods
based on unfitted grid. The unfitted FE method can date back to the
\cite{Babuska1970discontinuous}, which introduced a penalty term to
weakly enforce the jump on the interface. Li proposed the immersed FE
method in \cite{Li1998immersed}, which processes a better approximate
solution by modifying the basis functions near interface to capture
the jump of the solution. We refer to \cite{Na2014partially,
Lin2015partially, Adjerid2015immersed, Wang2009immersed,
Cao2017immersed, Guo2019higher} for some recent works.  Let us note
that the extended FE method is also a popular discretization method
\cite{Belytschko1999elastic}.

In 2002, A. Hansbo and P. Hansbo proposed an unfitted FE method with
the piecewise linear space and proved an optimal order of convergence
\cite{Hansbo2002unfittedFEM}. The numerical solution comes from
separate solutions defined on each subdomain and the jump conditions
are imposed weakly by Nitsche's method. Wadbro et al.
\cite{Wadbro2013uniformly} developed a uniformly well-conditioned FE
method based on Nitsche's method. Wu and Xiao \cite{Huang2017unfitted,
Wu2012unfitted} presented a $hp$ unfitted FE method, which is extended
to the three dimensional case. 
To achieve high-order accuracy and enjoy additional
flexibility, some authors tried to apply DG method to the elliptic
interface problem, for example the local DG method in
\cite{GuyomarcH2010discontinuous}, the hybridizable DG method in
\cite{Huynh2013hybrid} on fitted mesh, and the $hp$ DG method in
\cite{Massjung2012unfitted} on unfitted mesh.

Though high-order accuracy can be obtained, solid difficulties remain
for DG methods in solving problems with complex interfaces. To fit
curved interfaces, Cangiani et al. \cite{Cangiani2018adaptive}
introduced elements with curved faces to give an adaptive DG method
recently. As one of the latest work on unfitted mesh, Burman and Ern
\cite{Burman2018unfitted} proposed a hybrid high-order method, while
an extra assumption on the meshes are required to ensure the mesh
cells are cut favorably by the interface \cite{Massjung2012unfitted}.
In this paper, we are trying to propose a DG method on unfitted mesh
for the interface problem still using Nitsche's method. The novel
point is that we adopt a new approximation
space by patch reconstruction with one degree of freedom (DOF) per
element following the methodology in \cite{Li2016discontinuous,
  Li2017discontinuous}. The new space may be regarded as a subspace of
the approximation space used in \cite{Massjung2012unfitted}. Thanks to
the flexibility in choosing reconstruction patches, we may allow the
interface to intersect elements in a very general manner, in
comparison to the methods in \cite{Burman2018unfitted,
Massjung2012unfitted}. Following the standard DG discretization, the
elliptic interface problem is approximated by using a symmetric
interior penalty bilinear form with a Nitsche-type penalization at the
interface. The optimal error estimate is then derived in both DG
energy norm and $L^2$ norm. \revise{The patch reconstruction can
provide the stability near the interface and no cut-dependent
numerical flux is used. We note that the idea of using the patch of
interface elements to improve the numerical stability can also be
found in \cite{Guzman2017finite, Chu2010new, Huang2017unfitted,
Guo2019higher}}.  In addition, the classical DG methods for elliptic
problems were challenged \cite{hughes2000comparison,
zienkiewicz2003discontinuous} since it may use more DOFs than
traditional conforming FE methods. As a new observation, we
demonstrate by numerical examples that using our new approximation
space, one needs much less DOFs than classical DG methods. For
high-order approximations, number of DOFs can be even less than
conforming FE methods to achieve the same numerical error.

The rest of this paper is organized as follows. In Section
\ref{sec:space}, we introduce the reconstruction operator and the new
approximation space, and we also give the basic properties of the
approximation space. In Section \ref{sec:estimate}, the approximation
to the elliptic interface problem is proposed and we derive the
optimal error estimate in DG energy norm and $L^2$ norm. In Section
\ref{sec:numericalresults}, we present a lot of numerical examples to
verify the error estimate in Section \ref{sec:estimate}. To show the
performance of our method in efficiency, we make a comparison of
number of DOFs with respect to the numerical error between different
methods. We also solve a problem that admits solutions with low
regularities to illustrate the robustness of our method.


\section{Approximation Space}
\label{sec:space}
Let $\Omega \subset \mb{R}^d$, $d = 2$ or $3$, be a convex and
polygonal (polyhedral) domain with boundary $\partial \Omega $ and
let $\Gamma$ be a $C^2$-smooth interface  which divides $\Omega$ into
two open sets $\Omega_0$ and $\Omega_1$ satisfying $\Omega_0 \cap
\Omega_1 = \varnothing, \ \overline{\Omega} = \overline{\Omega}_0 \cup
\overline{\Omega}_1$ and $\Gamma = \overline{\Omega}_0 \cap
\overline{\Omega}_1$. We denote by $\MTh$ a partition of $\Omega$ into
polygonal (polyhedral) elements. Here we do not require the faces of
elements in $\MTh$ align with the interface (see Fig
\ref{fig:domain}).
\begin{figure}[htb]
  \centering
  \vspace{-10pt}
  \includegraphics[width=3in]{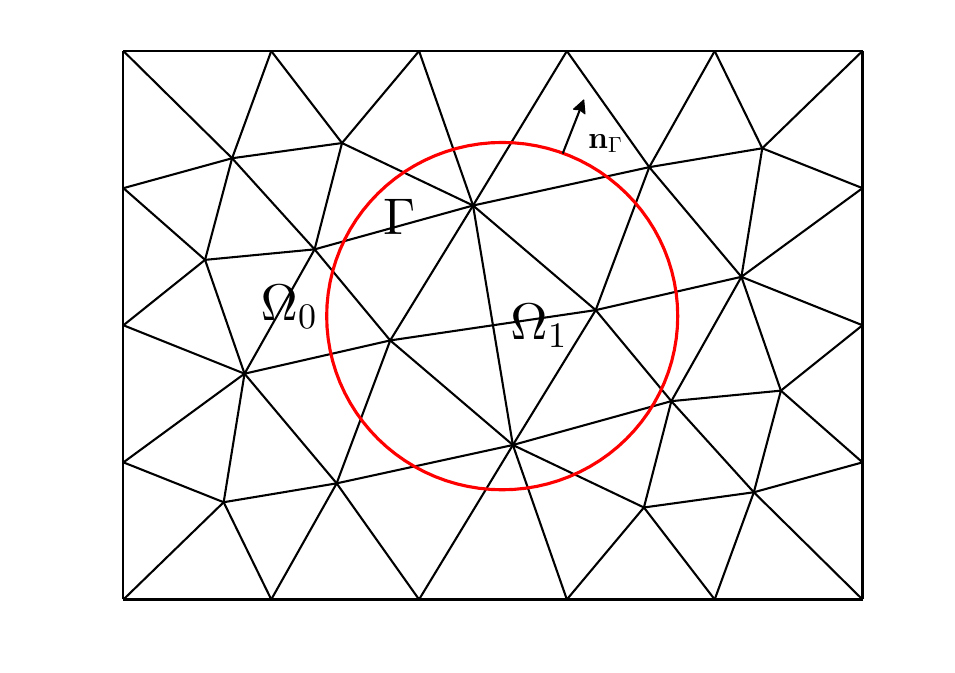}
  \vspace{-15pt}
  \caption{A sample domain and unfitted mesh for $d=2$.}
  \label{fig:domain}
\end{figure}
Let $\mc E^\circ_h$ be the set of all interior faces of $\MTh$, $\mc
E^b_h$ the set of the faces on $\partial \Omega$ and then $\mc E_h =
\mc E_h^\circ \cup \mc E_h^b$.  We set
\begin{displaymath}
  h_K = \text{diam}(K), \quad \forall K \in \MTh, \quad  h_e = |e| ,
  \quad \forall e \in \MEh,
\end{displaymath}
and we denote by $h$ the biggest one among the diameters of all
elements in $\MTh$. We assume that $\MTh$ is share-regular in the
sense of satisfying the conditions introduced in
\cite{Antonietti2013mimetic}, which are: there exist
\begin{itemize}
  \item two positive numbers $N$ and $\sigma$ which are
    independent of mesh size $h$;
  \item a compatible sub-decomposition $\widetilde{\mc{T}}_h$ into
    shape-regular triangles (tetrahedrons);
\end{itemize}
such that 
\begin{itemize}
  \item any polygon (polyhedron) $K \in \MTh$ admits a decomposition
    $\widetilde{\mathcal T}_{h|K}$ which has less than $N$
    shape-regular triangles (tetrahedrons);
  \item the share-regularity of $\widetilde K \in
    \widetilde{\mc{T}}_h$ follows \cite{Ciarlet2002finite}: the ratio
    between $h_{\widetilde K}$ and $\rho_{\widetilde K}$ is bounded by
    $\sigma$: $h_{\widetilde K}/\rho_{\widetilde K} \leq \sigma$ where
    $\rho_{\widetilde K}$ is the radius of the largest ball inscribed
    in $\widetilde K$.
\end{itemize}
The above regularity requirements could bring some useful consequences
which are trivial to verify \cite{Antonietti2013mimetic}:
\newcounter{counter}
\setcounter{counter}{1}
\begin{enumerate}
\item[M\arabic{counter}] there exists a constant $\rho_v$ that only
  depends on $N$ and $\sigma$ such that $\rho_v h_K \leq h_e$ for
  every element $K$ and every edge $e$ of $K$. 
  \addtocounter{counter}{1}
\item[M\arabic{counter}] there exists a constant $\rho_s$ that only
  depends on $N$ and $\sigma$ such that for every element $K$ the
  following holds true
  \begin{displaymath}
    \rho_s \max_{\widetilde K \in \Delta(K)} h_{\widetilde K} \leq
    h_K,
  \end{displaymath}
  where $\Delta(K) = \left\{ K' \in \MTh\ |\ K' \cap K \neq
    \varnothing \right\}$ is the collection of the elements touching
  $K$.
  \addtocounter{counter}{1}
\item[M\arabic{counter}] there exists a constant $\tau$ that only
  depends on $N$ and $\sigma$ such that for every element $K$, there
  is a disk (ball) inscribed in $K$ with center at the point $\bm z_K
  \in K$ and the radius $\tau h_K$.
  \addtocounter{counter}{1}
\item[M\arabic{counter}][{\it Trace inequality}] there exists a
  constant $C$ such that
  \begin{equation}
    \| v\|_{L^2(\partial K)}^2 \leq C\left( h_K^{-1}
      \|v\|_{L^2(K)}^2 + h_K \|\nabla v\|_{L^2(K)}^2 \right), \quad
    \forall v \in H^1(K).
    \label{eq:traceinequality}
  \end{equation}
  \addtocounter{counter}{1}
\item[M\arabic{counter}][{\it Inverse inequality}] there exists a
  constant $C$ such that 
  \begin{equation} \begin{aligned}
      \|\nabla v\|_{L^2(K)} &\leq C h_K^{-1} \|v \|_{L^2(K)}, \quad
      \forall v \in \mb P_m(K), \\
    \end{aligned}
    \label{eq:inverseinequality}
  \end{equation}
  where $\mb P_m(\cdot)$ denotes the polynomial space of degree less
  than $m$.
  \addtocounter{counter}{1}
\end{enumerate}
Let us note that throughout the paper, $C$ and $C$ with a subscript
are generic constants that may be different from line to line but are
independent of the mesh size $h$ \substitute{.}{\revise{and how the
interface cuts the mesh.}} Given a bounded domain $D \subset
\mb{R}^d$ and an integer $r \geq 0$, we would use the standard
notations and definitions for the spaces $H^r(D)$, $L^r(D)$ and their
corresponding inner products and norms. Then we will use the following
notations related to the partition:
\begin{displaymath}
   \begin{aligned}
     e^0 &= e \cap \Omega_0, &&\quad e^1 = e\cap \Omega_1, \quad
     \forall e \in \MEh, \\
     K^0 &= K \cap \Omega_0, &&\quad K^1  = K \cap \Omega_1, \quad
     \forall K \in \MTh,  \\
     (\partial K)^0 &= \partial K \cap \Omega_0, &&\quad  (\partial
     K)^1 = \partial  K \cap \Omega_1, \quad \forall K \in \MTh, \\
   \end{aligned}
   \end{displaymath}
     \begin{displaymath} 
       \begin{aligned}
     \MTh^0 &= \left\{ K \in \MTh\ |\ |K^0| > 0 \right\}, 
     \quad
     &&\MTh^1 = \left\{ K \in \MTh\ |\ |K^1| > 0 \right\}, \\
     \MEh^0 &= \left\{ e \in \MEh\ |\ |e^0| > 0 \right\}, 
     \quad
     &&\MEh^1 = \left\{ e \in \MEh\ |\ |e^1| > 0 \right\}. \\
   \end{aligned}
\end{displaymath}
Furthermore, we denote by $\MTh^\Gamma = \left\{ K \in \MTh\ |\  K
\cap \Gamma \neq \varnothing \right\}$ the set of the elements that
are divided by $\Gamma$ and by $\MEh^\Gamma = \left\{ e \in \MEh\ |\ e
\cap \Gamma \neq \varnothing \right\}$ the set of the faces that are
divided by $\Gamma$. We set $\MThB = \MTh \backslash \MTh^\Gamma$ and
$\MEhB = \MEh \backslash \MEh^\Gamma$. \substitute{For
element}{\revise{For an element}} $K \in \MTh^\Gamma$ we denote
$\Gamma_K = K \cap \Gamma$.

We make the following assumptions about the mesh, which are actually
easy to be fulfilled.
\begin{assumption}
  \substitute{We assume any interface slide in $\Gamma_K$ is
  contained entirely in one element, and we assume that for each
  element $K \in \MTh^\Gamma$ the interface $\Gamma$ intersects its
  boundary $\partial K$ twice and each open face at most once.}{
  \revise{For any face $e \in \MEhG$, the intersection $e \cap \Gamma$
  is simply connected; that is, $\Gamma$ does not intersect an
  interior face multiple times.}
  }
\end{assumption}

\begin{assumption}
  \revise{For any element $K \in \MThG$, there exists a line (plane)
  $\wt{\Gamma}_K$ and a smooth function $\psi$ that maps
  $\wt{\Gamma}_K$ onto $\Gamma_K$.}
\end{assumption}

\begin{figure}
  \centering
  \begin{minipage}[t]{0.3\textwidth}
    \begin{center}
      \begin{tikzpicture}[scale=2.5]
        \draw[thick, black] (-1, 0) -- (0, 0.5) -- (0, -0.5) --
        (-1, 0);
        \draw[thick, black] (-0.7, 0.35) to [out=10, in=100] (0.25,
        -0.3);
        \node at (0.2, 0.1) {$\Gamma$};
        \node at (-0.35, 0) {$K$};
      \end{tikzpicture}
    \end{center}
  \end{minipage}
  \begin{minipage}[t]{0.3\textwidth}
    \begin{center}
      \begin{tikzpicture}[scale=2.5]
        \draw[fill, gray] (-0.5, 0.15) to [out=-52, in = 160] (-0,
        -0.15) to [out=30, in = 200] (0.5, 0.15)
        to [out=160, in=20] (-0.5, 0.15);
        \draw[thick, black] (-0.15, -0.5) -- (-0.6, -0.2) -- (0, 0.5) --
        (0.6, -0.2) -- (-0.15, -0.5);
        \draw[thick, black] (0, 0.5) -- (-0.15, -0.5);
        \draw[thick, dashed, black] (-0.6, -0.2) -- (0.6, -0.2);
        \node at (0.35, 0.32) {$\Gamma$};
        \node at (0, -0.3) {$K$};
      \end{tikzpicture}
    \end{center}
  \end{minipage}
  \begin{minipage}[t]{0.3\textwidth}
    \begin{center}
      \begin{tikzpicture}[scale=2.5]
        \draw[fill, gray] (-0.3, 0.5) to [out=200, in=60] (-0.53,
        0.15) to [out=-30, in = 130] (0.3, -0.5) to [out = 50, in =
          250]
        (0.39, -0.1) to [out=130, in = -30] (-0.3, 0.5);
        \draw[thick, black] (-0.15, -0.5) -- (-0.6, -0.2) -- (0, 0.5) --
        (0.6, -0.2) -- (-0.15, -0.5);
        \draw[thick, black] (0, 0.5) -- (-0.15, -0.5);
        \draw[thick, dashed, black] (-0.6, -0.2) -- (0.6, -0.2);
        \node at (-0.53, 0.39) {$\Gamma$};
        \node at (-0.22, -0.3) {$K$};
      \end{tikzpicture}
    \end{center}
  \end{minipage}
  \caption{Examples of cut elements in two dimensions (left) / in
  three dimensions (middle and right).}
  \label{fig:cutelements}
\end{figure}

\begin{assumption}
  For any element $K \in \MThG$, there exist two elements $K^0_\circ,
  K^1_\circ \in \Delta(K)$ such that $K^0_\circ \subset \Omega^0$ and
  $K^1_\circ \subset \Omega^1$.
\end{assumption}

\begin{figure}[htp]
  \centering
  \begin{tikzpicture}[scale=12]
    \draw[fill, red] (0.498276269975913,0.407637982286526) --
    (0.578438529450671,0.450991528710196) --
    (0.498559923224131,0.504617017118447);
    \draw[black, line width=0.85pt] (0.38, 0.39) arc (270:350:0.25);
    \node[below right] at (0.5-0.008, 0.5-0.015) {$K$};
    \node[left] at(0.42, 0.5) {$\Omega_0$};
    \node[left] at(0.72, 0.5) {$\Omega_1$};
    \node[left] at(0.67, 0.58) {$\Gamma$};
    \node[] at(0.47, 0.55) {$K_\circ^0$};
    \node[] at(0.552, 0.40) {$K_\circ^1$};
    \draw[thick, black] (0.419151595178986,0.450734782421807) --
    (0.425396419863252,0.355202235860487);
    \draw[thick, black] (0.498206885483482,0.329198010463432) --
    (0.425396419863252,0.355202235860487);
    \draw[thick, black] (0.498276269975913,0.407637982286526) --
    (0.425396419863252,0.355202235860487);
    \draw[thick, black] (0.498276269975913,0.407637982286526) --
    (0.419151595178986,0.450734782421807);
    \draw[thick, black] (0.413298247684544,0.546759345510935) --
    (0.419151595178986,0.450734782421807);
    \draw[thick, black] (0.498276269975913,0.407637982286526) --
    (0.498206885483482,0.329198010463432);
    \draw[thick, black] (0.571497373411646,0.354810571526876) --
    (0.498206885483482,0.329198010463432);
    \draw[thick, black] (0.498276269975913,0.407637982286526) --
    (0.571497373411646,0.354810571526876);
    \draw[thick, black] (0.498559923224131,0.504617017118447) --
    (0.498276269975913,0.407637982286526);
    \draw[thick, black] (0.498559923224131,0.504617017118447) --
    (0.419151595178986,0.450734782421807);
    \draw[thick, black] (0.498559923224131,0.504617017118447) --
    (0.413298247684544,0.546759345510935);
    \draw[thick, black] (0.578438529450671,0.450991528710196) --
    (0.498276269975913,0.407637982286526);
    \draw[thick, black] (0.578438529450671,0.450991528710196) --
    (0.571497373411646,0.354810571526876);
    \draw[thick, black] (0.578438529450671,0.450991528710196) --
    (0.498559923224131,0.504617017118447);
    \draw[thick, black] (0.498559923224131,0.504617017118447) --
    (0.498256744819994,0.620448915560764);
    \draw[thick, black] (0.498256744819994,0.620448915560764) --
    (0.413298247684544,0.546759345510935);
    \draw[thick, black] (0.65771661788519,0.395308762259631) --
    (0.571497373411646,0.354810571526876);
    \draw[thick, black] (0.578438529450671,0.450991528710196) --
    (0.65771661788519,0.395308762259631);
    \draw[thick, black] (0.585303658150289,0.548914734630151) --
    (0.578438529450671,0.450991528710196);
    \draw[thick, black] (0.585303658150289,0.548914734630151) --
    (0.498559923224131,0.504617017118447);
    \draw[thick, black] (0.585303658150289,0.548914734630151) --
    (0.498256744819994,0.620448915560764);
    \draw[thick, black] (0.660936451912223,0.493780615176984) --
    (0.578438529450671,0.450991528710196);
    \draw[thick, black] (0.660936451912223,0.493780615176984) --
    (0.65771661788519,0.395308762259631);
    \draw[thick, black] (0.660936451912223,0.493780615176984) --
    (0.585303658150289,0.548914734630151);
  \end{tikzpicture}
  \caption{The collection $\Delta(K)$, $K_\circ^0$ and $K_\circ^1$ .}
  \label{fig:collection}
\end{figure}
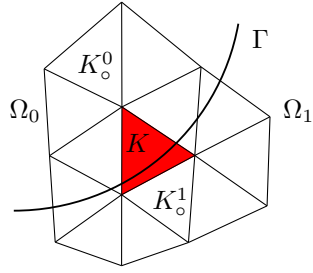
\revise{We note that Assumption 1 and Assumption 2 ensure the
interface $\Gamma$ is well resolved by the mesh
\cite{Massing2014stabilized} and such similar geometric assumptions are
commonly used in numerically solving interface problems
\cite{Massjung2012unfitted, Hansbo2002unfittedFEM, Wu2012unfitted,
Wadbro2013uniformly, Burman2018unfitted}. In Fig.
\ref{fig:cutelements}, we present some examples of cut elements to
illustrate the assumptions. 
}

For the given partition $\MTh$, we follow the idea in
\cite{Li2016discontinuous, Li2017discontinuous} to define the
reconstruction operator for solving the elliptic interface
problem. First, for every element $K \in \MTh$, we specify its
barycenter $x_K$ as a sampling point. Second, for each element
$K \in \MTh^i \backslash \MThG (i = 0, 1)$, we will construct an
element patch $S^i(K)$ for $K$. The element patch $S^i(K)$ is
a set of elements and consists of $K$ and some elements around $K$.
We start from assigning a threshold value $\# S(K)^i$ that is used to
control the size of $S^i(K)$, and we setup the element patch $S^i(K)$
in a recursive manner. Let $S_0^i(K) = \left\{ K \right\}$, then we
define $S^i_t(K)$ as
\begin{displaymath}
  S_t^i(K) = \bigcup_{\tiny
  \begin{aligned}
    \widetilde{K} \in \MTh^i ,& \widehat{K} \in S_{t - 1}^i(K) \\
    \widetilde{K} \cap \widehat{K} &= e \in \mathcal E_h
  \end{aligned}
} \widetilde{K}, \quad t = 1, 2, \cdots 
\end{displaymath}
Once $S^i_t(K)$ has collected $\# S^i(K)$ elements, we stop the
procedure and let $S^i(K) = S^i_t(K)$. Clearly, the cardinality of
$S^i(K)$ is just the value $\# S^i(K)$. For any element $K \in \MThB$,
we only construct one element patch which satisfies that if $K \in
\MTh^i \backslash \MThG$ then $S^i(K) \subset \MTh^i$.  For any
element $K \in \MThG$, we assume that $K \in S^0(K_\circ^0)$ and $K
\in S^1(K_\circ^1)$ where $K^0_\circ$ and $K^1_\circ$ are defined in
Assumption 2. With $\# S^0(K_\circ^0)$ and $\# S^1(K_\circ^1)$ to be
mildly greater than that in \cite{Li2016discontinuous,
Li2017discontinuous}, the assumption can be fulfilled according to the
method of constructing the element patch. Consequently, for each
element $K \in \MThG$, we have two element patches $S^0(K) =
S^0(K_\circ^0)$ and $S^1(K) = S^1(K_\circ^1)$. \substitute{Some
examples are presented in Appendix \ref{sec:cep} to illustrate the
construction of the element patch.}{\revise{In Appendix \ref{sec:cep},
we present the detailed algorithm and give some examples to illustrate
the construction of the element patch.}}

For any element $K \in \MTh$, we denote by $\mc I_K^i(i= 0 , 1)$ the
set of sampling points located inside $S^i(K)$,
\begin{displaymath}
  \mc I_K^i =  \left\{ \bm x_{\widetilde K} \ |\ \forall \widetilde K
  \in S^i(K)\right\}.
\end{displaymath}
For any function $g \in C^0(\Omega)$ and an element $K \in \MTh$, we
seek a polynomial $\mc R_K^i g$ defined on $S^i(K)$ of degree $m$ by
solving the following least squares problem:
\begin{equation}
  \mc R_K^i g = \mathop{\arg \min}_{p \in \mathbb P_m(S^i(K))} \sum_{
  \bm x \in \mc I^i_K} |p(\bm x) - g(\bm x)|^2.
  \label{eq:leastsquare}
\end{equation}
The existence and uniqueness of the solution to \eqref{eq:leastsquare}
are decided by the position of the sampling nodes in $\mc I_K^i$. Here
we follow \cite{Li2012efficient} to make the following assumption:
\begin{assumption}
  For any element $K \in \MTh$ and $p \in \mathbb P_m(S^i(K))$, 
  \begin{displaymath}
    p|_{\mc I^i_K} = 0 \quad \text{implies} \quad p|_{S^i(K)} \equiv 0,
    \quad i = 0, 1.
  \end{displaymath}
\end{assumption}
This assumption actually rules out the situation that all the points
in $\mc I^i_K$ are located on an algebraic curve of degree
$m$. Definitely, this assumption requires the cardinality $\# S^i(K)$
shall be greater than $\text{dim}(\mathbb P_m)$. Hereafter, we always
require this assumption holds.

Since the solution to \eqref{eq:leastsquare} is linearly dependent on
$g$, we define two interpolation operators $\mc R^i$ for $g$:
\begin{displaymath}
  \begin{aligned}
    (\mc R^0 g)|_K &= (\mc R^0_K g)|_K, \quad \text{for } \forall K
    \in \MTh^0, \\ 
    (\mc R^1 g)|_K &= (\mc R^1_K g)|_K, \quad \text{for } \forall K
    \in \MTh^1.  \\
  \end{aligned}
\end{displaymath}
Given $\mc R^i(i= 0, 1)$ and $g \in C^0(\Omega)$, the function $g$ is
mapped to a piecewise polynomial function of degree $m$ on
$\MTh^i$. We denote by $V_h^i$ the image of the operator $\mc
R^i$. \revise{For any element $K$, We pick up a function $w_K^i(\bm
x) \in C^0(\Omega)$ such that}
\begin{displaymath}
  w_K^i(\bm x) = \begin{cases}
    1, \quad \bm x = \bm x_K, \\
    0, \quad \bm x \in \widetilde K, \quad \widetilde K \neq K.
  \end{cases} 
\end{displaymath}
\revise{It should be noted that in element $K$ we do not care about
the values of $w_K^i(\bm{x})$ at $\bm{x} \neq \bm{x}_K$ and such
continuous functions obviously exist. }
Then it is easy to check that $V_h^i = \mathrm{span} \left\{
\lambda_K^i\ |\ \lambda^i_K = \mc R^i w_K^i \right\}$, and one can
write the operator $\mc R^i$ in an explicit way:
\begin{displaymath}
  \mc R^i g = \sum_{K \in \MTh^i} g(\bm x_K) \lambda_K^i(\bm x), \quad
  \forall g \in C^0(\Omega).
\end{displaymath}
In Appendix \ref{sec:1dexample}, we present a one-dimensional example
to show more details of construction of $\lambda_K^i$ and its computer
implementation.

\begin{remark}
  \revise{
  The computational cost of constructing the approximation spaces
  $V_h^0$ and $V_h^1$ mainly consists of two parts. The first is the
  construction of element patches. We adopt a recursive strategy on
  every element for construction as we illustrate in Appendix
  \ref{sec:cep}. The number of recursive steps is related to the order
  $m$ and in numerical experiments we take $1 \leq m \leq 3$. We also
  list the values $\# S^i(K)$ that are used in numerical experiments
  in Section \ref{sec:numericalresults}.  In this case at most $6$
  recursive steps are required on each element.  Hence this part is
  very cheap.  The second part is to solve the function $\lambda_K^i$
  on each element. In this part, the main step is to solve an inverse
  of a $\# S^i(K) \times \# S^i(K)$ matrix, as we demonstrate in
  Appendix \ref{sec:1dexample}.  Thus, the computational cost of the
  second part is still small.
  }
\end{remark}

The operators $\mc R^i(i = 0, 1)$ are defined for functions in
$C^0(\Omega)$, while we only concern the case for the functions in
$H^t(\Omega_0 \cup \Omega_1)(t \geq 2)$. 
Hence, we \substitute{chosse}{\revise{choose}} two extension operators
to extend the functions in $H^t(\Omega_0 \cup \Omega_1)$ to be defined
in $H^t(\Omega)$ \cite{Adams2003sobolev}. For any function $w \in
H^t(\Omega_0 \cup \Omega_1)$, there exist two operators $E^i:
H^t(\Omega_i) \rightarrow H^t(\Omega)$ such that $(E^i w)|_{\Omega_i}
= w$ and
\begin{equation}
  \|E^iw\|_{H^s(\Omega)} \leq C \|w\|_{H^s(\Omega_i)}, \quad 0 \leq s
  \leq t.
  \label{eq:extension}
\end{equation}

Now let us study the approximation property of the operator $\mc R^i$.
We define $\Lambda(m, S^i(K))$ for all element patches as
\begin{displaymath}
  \Lambda(m, S^i(K)) = \max_{p \in \mb P_m(S^i(K))}
  \frac{\max_{\bm x \in S^i(K)}|p(\bm x)|}{\max_{\bm x \in \mc I_K^i}
  |p(\bm x)|}.
\end{displaymath}
We note that under some mild conditions on $S^i(K)$, $\Lambda(m,
S^i(K))$ admits a uniform upper bound $\Lambda_m$, which is crucial in
the convergence analysis. \substitute{ We refer to
\cite{Li2012efficient, Li2016discontinuous} for the conditions and
more discussion about the uniform upper bound.  Here, we point out
that one of the conditions is the value $\# S^i(K)$ should be greater
than $\text{dim}(\mathbb P_m)$. In Section \ref{sec:numericalresults},
we list the values of $\# S^i(K)$ for different $m$ in all numerical
examples.}{\revise{We refer to \cite[Assumption
A]{Li2016discontinuous} for the geometrical conditions on element
patches. These conditions in fact exclude the case that the points in
$\mc{I}_K^i$ are very close to an algebraic curve of degree $m$.
We also proved that if the size of the element
patch $S^i(K)$ is greater than a certain number, then the geometrical
conditions will be satisfied, see  \cite[Lemma 6]{Li2016discontinuous}
and \cite[Lemma 3.4]{Li2012efficient}. We note that that this number
is usually too great and we prefer not to adopt it in the
implementation. In numerical tests, we observe that out method can
still work very well under the case that the value $\# S^i(K)$ is far
less than the theoretical value. In Section
\ref{sec:numericalresults}, we list the values of $\# S^i(K)$ that are
used in numerical examples. In addition, we refer to
\cite{Li2012efficient} for some numerical experiments about the size
of the element patch and the upper bound $\Lambda_m$.
}}

\begin{remark}
  \revise{ For a special case when $\# S^i(K) =
  \text{dim}(\mb{P}_m)$, we may replace the constant $\Lambda(m,
  S^i(K))$ by the Lebesgue constant
  \cite[p.24]{Powell1981approximation}. In this case, the solution to
  the problem \eqref{eq:leastsquare} is the Lagrange interpolation
  polynomial. Unfortunately, we have little knowledge of the Lebesgue
  constant in two or three dimensions.
  }
\end{remark}


With $\Lambda_m$, we have the local approximation error estimates. 
\begin{theorem}
  Let $g \in H^t(\Omega_0 \cup \Omega_1)(t \geq 2)$, there exist
  constants $C$ such that for any $K \in \MTh^i(i = 0, 1)$ the
  following estimates hold true:
  \begin{equation}
    \begin{aligned}
      \|E^ig - \mc R^i(E^ig)\|_{H^q(K)} &\leq C \Lambda_m h_K^{s -
      q} \| E^ig\|_{H^t(S^i(K))}, \quad q = 0, 1,\\
      \| \substitute{D}{\nabla}^q(E^ig - \mc
      R^i(E^ig))\|_{L^2(\partial K )} & \leq C\Lambda_m h_K^{s - q-
      1/2} \|E^ig\|_{H^t(S^i(K))}, \quad q = 0, 1,\\
    \end{aligned}
    \label{eq:interpolation}
  \end{equation}
  where $s = \min(t + 1, m)$.
\end{theorem}
\begin{proof}
  It is a direct consequence of \cite[lemma 2.4]{Li2016discontinuous}
  or \cite[lemma 2.5]{Li2017discontinuous}.
\end{proof}

Finally, we give the definition of our approximation space $V_h$ by
concatenaing the two spaces $V_h^0$ and $V_h^1$. Let us define a
global interpolation operator $\mc R$: for any function
$w \in H^t(\Omega_0 \cup \Omega_1)$, $\mc Rw$ is piecewise defined by
\begin{displaymath}
  (\mc Rw)|_K \triangleq \begin{cases}
    (\mc R_K^0 E^0 w)|_K, \quad \text{for } K \in \MTh^0 \backslash
    \MThG,
    \\
    (\mc R_K^1 E^1 w)|_K, \quad \text{for } K \in \MTh^1 \backslash
    \MThG, \\
    (\mc R_K^i E^i w)|_{K^i}, \quad \text{for } K \in \MThG, \quad i =
    0, 1. \\
  \end{cases}
\end{displaymath}
The image of $\mc R$ is actually our new approximation space $V_h$. We
notice that for any function $w \in H^t(\Omega_0 \cup \Omega_1)$,
$\mc Rw$ is a combination of $\mc R^0 w$ and $\mc R^1w$ that
$(\mc Rw)|_{K^i} = (\mc R^iw)_{K^i}(i = 0, 1)$, and the approximation
error estimates of $\mc R$ are the direct consequence from
\eqref{eq:interpolation}.


\section{Approximation to Elliptic Interface Problem}
\label{sec:estimate}
We consider the standard elliptic interface problem: find $u$ in
$H^2(\Omega_0 \cup \Omega_1)$ such that
\begin{equation}
  \begin{aligned}
    -\nabla \cdot \beta \nabla u &= f ,\quad \bm x \in \Omega_0 \cup
    \Omega_1, \\
    u &= g, \quad \bm x \in \partial \Omega, \\
    \jump{u} &= a \bm{\mr n}_\Gamma, \quad \bm x \in \Gamma, \\
    \jump{\beta \nabla u\cdot  \bm{\mr n}_\Gamma} &= b \bm{\mr
    n}_\Gamma, \quad \bm x \in \Gamma, \\
  \end{aligned} 
  \label{eq:elliptic}
\end{equation}
where $\beta$ is a positive constant function on $\Omega_i(i = 0, 1)$
but may be discontinuous across the interface $\Gamma$, and $\bm{\mr
n}_\Gamma$ denotes the \substitute{outward unit normal to
$\Gamma$.}{\revise{the unit normal of $\Gamma$ pointing to
$\Omega_0$ (see Fig \ref{fig:domain}).}} The source term $f$, the
Dirichlet data $g$ and the jump term $a$, $b$ are assumed to be in
$L^2(\Omega)$, $H^{3/2}(\partial \Omega)$, $H^{3/2}(\Gamma)$,
$H^{1/2}(\Gamma)$, respectively, to ensure \eqref{eq:elliptic} has a
unique solution. We refer to \cite{Roitberg1969theorem,
Kellogg1975poisson, Petzoldt2001regularity, Kellogg1972higher} for
more details. In \eqref{eq:elliptic}, the jump operator $\jump{\cdot}$
takes the standard sense in DG framework. More precisely, we define
the jump operator $\jump{\cdot}$ and average operator $\aver{\cdot}$
as below,
\begin{displaymath}
  \begin{aligned}
    \jump{\bm{q}} = \begin{cases}
      \bm{q}|_{K_+} \cdot \bm{\mr{n}}_{K_+} +  \bm{q}|_{K_-} \cdot
      \bm{\mr{n}}_{K_-} \\ 
      \bm{q}|_{K_+^i} \cdot \bm{\mr{n}}_{K_+} + \bm{q}|_{K_-^i} \cdot
      \bm{\mr{n}}_{K_-} \\
      \bm{q}|_{K} \cdot \bm{\mr{n}}_K  \\
      (\bm{q}|_{K^1} - \bm{q}|_{K^0}) \cdot \bm{\mr n}_\Gamma \\
    \end{cases}
    \jump{v} &= \begin{cases}
      v|_{K_+}\bm{\mr{n}}_{K_+} +  v|_{K_-}\bm{\mr{n}}_{K_-} &\quad 
      \text{on } e
      \in \MEh^\circ \backslash \MEh^\Gamma, \\
      v|_{K_+^i}\bm{\mr{n}}_{K_+} +  v|_{K_-^i}\bm{\mr{n}}_{K_-}
      &\quad \text{on } e \in \MEh^\Gamma \cap \Omega_i(i = 0,
      1),\\
      v|_{K}\bm{\mr{n}}_K & \quad \text{on } e \in \MEh^b, \\
      (v|_{K^1} - v|_{K^0})\bm{\mr n}_\Gamma &\quad \text{on }
      \Gamma_K, \ K \in \MTh^\Gamma, \\
    \end{cases} \\
  \aver{\bm{q}} = \begin{cases}
    \frac{1}{2} (\bm{q}|_{K_+} + \bm{q} |_{K_-}) \\
    \frac12(\bm{q}|_{K_+^i} +  \bm{q} |_{K_-^i}) \\
    \bm{q}|_K \\
    \frac12(\bm{q}|_{K^1} + \bm{q}|_{K^0}) \\
  \end{cases} \hspace{36pt}
  \aver{v}& = \begin{cases}
    \frac{1}{2} (v|_{K_+} + v|_{K_-}) \hspace{25pt} &\quad
    \text{on } e \in \MEh^\circ \backslash \MEh^\Gamma, \\
    \frac12(v|_{K_+^i} +  v|_{K_-^i}) &\quad \text{on } e \in
    \MEh^\Gamma \cap \Omega_i(i = 0, 1), \\
    v|_{K} & \quad \text{on } e \in \MEh^b, \\
    \frac12(v|_{K^1} + v|_{K^0}) &\quad \text{on } \Gamma_K,\  K
    \in \MTh^\Gamma, \\
  \end{cases} \\
  \end{aligned}
\end{displaymath}
where $v$ is a scalar-valued function and $\bm{q}$ is a vector-valued
function. For $e \in \MEh^\circ$, we let $K_+$ and $K_-$ be two
neighbouring elements that share a common face $e$. $\bm{\mr n}_{K_+}$
and $\bm{\mr n}_{K_-}$ are the unit outer normal on $e$ corresponding
to $\partial K_+$ and $\partial K_-$, respectively. In the case $e \in
\MEh^b$, we let $e$ be a face of the element $K$.


Now we define the bilinear form $b_h(\cdot, \cdot)$ and the linear
form $l_h(\cdot)$:
\begin{equation}
  \begin{aligned}
    b_h(u_h, v_h) & =  \sum_{K \in \MTh} \int_{K^0 \cup K^1} \beta
    \nabla u_h \cdot \nabla v_h \d{x} \\
    &- \left[ \sum_{e \in \MEh} \int_{e^0 \cup e^1} + \sum_{K \in
        \MTh^\Gamma} \int_{\Gamma_K} \right] \Big( \jump{u_h} \cdot
    \aver{\beta \nabla v_h} + \jump{v_h} \cdot \aver{\beta
      \nabla u_h} \Big) \d{s} \\
    &+  \sum_{e \in \MEh} \int_{e^0 \cup e^1}
    \frac{\eta}{h_e} \jump{u_h} \cdot \jump{v_h} \d{s} 
    +  \sum_{K \in \MThG} \int_{\Gamma_K}
    \frac{\eta}{h_K} \jump{u_h} \cdot \jump{v_h} \d{s}, \\
  \end{aligned}
  \label{eq:bilinear}
\end{equation}
for $\forall u_h, v_h \in W_h$, and
\begin{displaymath}
  \begin{aligned}
    l_h(v_h)&  = \sum_{K \in \MTh} \int_{K^0 \cup K^1} fv_h \d{x} -
    \sum_{e \in \MEh^b} \int_e g \un \cdot \aver{\beta \nabla v_h}
    \d{s} \\
    & + \sum_{\revise{K \in \MThG}}  \int_{\Gamma_K} b \aver{v_h}
    \d{s} - \sum_{\revise{K \in \MThG}} \int_{\Gamma_K} a \un_\Gamma
    \cdot \aver{\beta \nabla v_h} \d{s}\\ 
    &+ \sum_{e \in \MEh^b} \int_e \frac{\eta}{h_e}  g v_h \d{s} +
    \sum_{\revise{K \in \MThG}} \int_{\Gamma_K} \frac{\eta}{h_K} a
    \un_\Gamma \cdot \jump{v_h} \d{s}, \\
  \end{aligned}
\end{displaymath}
for $\forall v_h \in W_h$, where $W_h$ denotes the following broken
Sobolev space 
\begin{displaymath}
  \begin{aligned}
    W_h = \Big \{ v \in L^2(\Omega)\ \Big|\ &v|_K \in H^2(K),\ \text{for
    } K \in \MThB,\\
  \quad &v|_{K^i} \in H^2(K^i), \  i = 0, 1, \ \text{for } K \in \MThG
  \Big \}. \\
\end{aligned}
\end{displaymath}
The penalty parameter $\eta$ is nonnegative and will be specified
later on. For any $v_h \in W_h$, let us define a DG energy norm
$\enorm{ \cdot}$ as
\begin{displaymath}
  \begin{aligned}
    \enorm{v_h}^2 = \|\nabla v_h\|_{L^2(\MTh^0 \cup \MTh^1)}^2 & +
    \|h_e^{-1/2}\jump{v_h}\|_{L^2(\MEh^0 \cup \MEh^1)}^2 + \|h_e^{1/2}
    \aver{\nabla v_h}\|_{L^2(\MEh^0 \cup \MEh^1)}^2 \\
    & + \|h_K^{-1/2} \jump{v_h}\|_{L^2(\Gamma)}^2 + \|h_K^{1/2}
    \aver{\nabla v_h}\|_{L^2(\Gamma)}^2,  \\
  \end{aligned}
\end{displaymath}
where
\begin{displaymath}
  \begin{aligned}
    \| \nabla v_h \|_{L^2(\MTh^0 \cup \MTh^1)}^2 = \sum_{K \in \MTh}
    \int_{K^0 \cup K^1} | \nabla v_h |^2  &\d{x}, \quad \|h_e^{-1/2}
    \jump{v_h} \|_{L^2(\MEh^0 \cup \MEh^1)}^2 = \sum_{e \in \MEh}
    \int_{e^0 \cup e^1} \frac{1}{h_e} | \jump{v_h} |^2 \d{s}, \\
    \|h_e^{1/2} \aver{\nabla v_h}\|_{L^2(\MEh^0 \cup \MEh^1)}^2 =
    \sum_{e \in \MEh} \int_{e^0 \cup e^1} h_e | \aver{\nabla v_h}|^2
    &\d{s}, \quad \|h_K^{-1/2} \jump{v_h}\|_{L^2(\Gamma)}^2 =\sum_{K
    \in \MTh^\Gamma} \int_{\Gamma_K}\frac{1}{h_K}  |\jump{v_h} |^2
    \d{s}, \\
    \| h_K^{1/2} \aver{\nabla v_h}\|_{L^2(\Gamma)}^2 &= \sum_{K \in
    \MTh^\Gamma} \int_{\Gamma_K} h_K |\aver{\nabla v_h}|^2 \d{s}. \\
  \end{aligned}
\end{displaymath}

The approximation problem to the elliptic interface problem
\eqref{eq:elliptic} is then defined as: {\it find $u_h \in V_h$ such
that}
\begin{equation}
  b_h(u_h, v_h) = l_h(v_h), \quad \forall v_h \in V_h.
  \label{eq:discrete}
\end{equation}
An immediate consequence from the definitions of \substitute{the
bilinear $b_h(\cdot, \cdot)$ and $l_h(\cdot)$}{\revise{the bilinear
form $b_h(\cdot, \cdot)$ and the linear form $l_h(\cdot)$}} is the
validity of the Galerkin orthogonality, which plays a key role in the
error estimate later on.
\begin{lemma}
  Let $u \in H^2(\Omega_0 \cup \Omega_1)$ be the exact solution and
  let $u_h \in V_h$ be the solution to \eqref{eq:discrete}, the
  Galerkin orthogonality holds true: 
  \begin{equation}
    b_h(u - u_h, v_h) = 0, \quad \forall v_h \in V_h.
    \label{eq:orthogonality}
  \end{equation}
  \label{le:orthogonality}
\end{lemma}
\begin{proof}
  By $\jump{u} = 0$ on any $e \in \MEh^\circ$ and $\jump{u} = g \un$
  on any $e \in \MEh^b$, we observe that
  \begin{displaymath}
    \begin{aligned}
      b_h(u, v_h) = \sum_{K \in \MTh} \int_{K^0 \cup K^1} \beta \nabla
      u \cdot \nabla v_h \d{x} & - \sum_{e \in \MEh} \int_{e^0 \cup
      e^1} \jump{v_h} \cdot \aver{\beta \nabla u} \d{s} - \sum_{K \in
      \MThG} \int_{\Gamma_K} \jump{v_h} \cdot \aver{\beta \nabla u}
      \d{s} \\
      & - \sum_{e \in \MEh^b} \int_e g \un \cdot \aver{\beta \nabla
      v_h} \d{s} - \sum_{K \in \MThG} \int_{\Gamma_K}
      a \un_{\Gamma} \cdot \aver{\beta \nabla v_h} \d{s} \\
      & + \sum_{e \in \MEh^b} \int_e \frac{\eta}{h_e} g v_h \d{s} +
      \sum_{K \in \MThG} \int_{\Gamma_K} \frac{\eta}{h_K} a
      \un_{\Gamma} \cdot \jump{v_h} \d{s}. \\
    \end{aligned}
  \end{displaymath}
  Applying integration by parts, we have that 
  \begin{displaymath}
    \begin{aligned}
      \sum_{K \in \MTh} \int_{K^0 \cup K^1} \beta \nabla u \cdot
      \nabla v_h \d{x} & = - \sum_{K \in \MTh} \int_{K^0 \cup K^1}
      \nabla \cdot(\beta \nabla u) v_h \d{x} + \sum_{e \in \MEh}
      \int_{e^0 \cup e^1} \jump{v_h} \cdot (\beta \nabla u) \d{s} \\
      & + \sum_{K \in \MThG} \int_{\Gamma_K} \jump{v_h} \cdot \aver{
      \beta \nabla u} \d{s} + \sum_{K \in \MThG} \int_{\Gamma_K} b
      \aver{v_h} \d{s}. \\
    \end{aligned}
  \end{displaymath}
  Combining above two equations implies $b_h(u_h, v_h) = b_h(u, v_h)$,
  which completes the proof.
\end{proof}

Next we verify the boundedness and coercivity of the bilinear form
$b_h(\cdot, \cdot)$ with respect to the energy norm $\enorm{\cdot}$.
For this purpose, we need to estimate the error on the interface. Here
we first give the discrete trace inequality, which is crucial in the
error estimate.
\begin{lemma}
  For any $K \in \MThG$, there exists a constant $C$ such that 
  \begin{equation}
    \|\nabla^\alpha v_h \|_{L^2(\partial K^i)} \leq C h_K^{-1/2}
    \|\nabla^\alpha v_h\|_{L^2(K_\circ^i)}, \quad \forall v_h \in V_h,
    \quad i = 0, 1, \quad \alpha = 0, 1, 
    \label{eq:traceestimate}
  \end{equation}
  where $\partial K^i = (\partial K)^i \cup \Gamma_K$.
  \label{le:traceestimate}
\end{lemma}
\begin{proof}
  Since $K \in \MThG$, we have that the patch $S^i(K)$ is the same as
  the patch $S^i(K_\circ^i)$. From the definition of the least squares
  problem \eqref{eq:leastsquare}, it is clear that the solution to
  \eqref{eq:leastsquare} on $S^i(K)$ is the same as the solution to
  \eqref{eq:leastsquare} on $S^i(K_\circ^i)$.  Particularly,
  $\nabla^\alpha v_h|_{K^i}$ and $\nabla^\alpha v_h|_{K_\circ^i}$ are
  exactly the same polynomial which is denoted as $\tilde{p}$.
  \substitute{Based on M3, there exists a constant $\hat{\tau}$ such
  that $B(\bm z_{K_\circ^i}, \hat{\tau} h_{K_\circ^i}) \subset
  K_\circ^i$, where $B(\bm z, r)$ is a ball with center at $\bm z$ and
  radius $r$.  Owing to M2 and Assumption 2, there exists a constant
  $\tilde{\tau}$ such that $\partial K^i \subset B(\bm z_{K_\circ^i},
  \tilde{\tau} h_K)$, and we observe that}{\revise{Based on M3, there
  exists a constant $\hat{\tau}$ such that $B(\bm z_{K_\circ^i},
  \hat{\tau} h_{K_\circ^i}) \subset K_\circ^i$, where $B(\bm z, r)$ is
  a ball with center at $\bm z$ and radius $r$. From Assumption 2, we
  have that $K \in \Delta(K_\circ^i)$. By the mesh regularity M2,
  there exists a constant $\tilde{\tau}$ such that $\partial K^i
  \subset B(\bm z_{K_\circ^i}, \tilde{\tau} h_{K_\circ^i})$ and there
  exists a constant $C$ such that $h_{K} \leq C h_{K_\circ^i}$. We
  note that here the constants $\hat{\tau}$, $\tilde{\tau}$ and $C$
  only depend on $N$ and $\sigma$. We further deduce that}}
  \begin{displaymath}
    \begin{aligned}
      \|\tilde{p}\|_{L^2(\partial K^i)} &\leq |\partial K^i|^{\frac12}
      \|\tilde{p}\|_{L^\infty(\partial K^i)}  \leq |\partial
      K^i|^{\frac12} \|\tilde{p}\|_{L^\infty(B(\bm z_{K_\circ^i},
      \tilde{ \tau} h_{K_\circ^i}))} \\
      & \leq C |\partial K^i|^{\frac12} |B(\bm z_{K_\circ^i},
      \tilde{\tau} h_{K_\circ^i} )|^{-\frac12}
      \|\tilde{p}\|_{L^2(B(\bm z_{K_\circ^i}, \hat{\tau} h_{K_\circ^i}
      ))} \\ & \leq C |\partial K^i|^{\frac12} |B(\bm z_{K_\circ^i},
      \tilde{\tau} h_{K_\circ^i} )|^{-\frac12}
      \|\tilde{p}\|_{L^2(K_\circ^i)} \\
      & \leq Ch_K^{\frac{d-1}{2}} h_{K_\circ^i}^{-\frac{d}{2}}
      \|\tilde{p}\|_{L^2(K_\circ^i)} \leq C h_K^{-\frac12}
      \|\tilde{p}\|_{L^2(K_\circ^i)}.\\
    \end{aligned}
  \end{displaymath}
  The third inequality follows from the inverse inequality $\|\hat{p}
  \|_{L^\infty(B(0, 1))} \leq C \|\hat{p}\|_{L^2(B(0, \hat{\tau} /
  \tilde{\tau}))}$ for any $\hat{p} \in \mathbb P_m(B(0, 1))$ and the
  pullback using the bijective affine map from $B(\bm z_{K_\circ^i},
  \tilde{\tau} h_{K_\circ^i} )$ to $B(0, 1)$. As $\Gamma$ is of class
  $C^2$, it is easy to show (cf. \cite{Chen1998interface,
  Wu2012unfitted}) $|\Gamma_K| \leq C h_K^{d - 1}$. We complete the
  proof by observing $|\partial K^i| \leq h_K^{d - 1}$ and
  $|B(\bm{z}_{K_\circ^i}, \tilde{\tau} h_{K_\circ^i} )| \leq C
  h_{K_\circ^i}^d$.
\end{proof}

\begin{lemma}
  There exists a positive constant $h_0$ independent of $h$ and the
  location of the interface such that for all $h \leq h_0$ and any
  element $K \in \MThG$, the following trace inequality holds true:
  \begin{equation}
    \|w\|_{L^2(\Gamma_K)}^2 \leq C \left( h_K^{-1} \|w\|_{L^2(K)}^2 +
    h_K \|\nabla w\|_{L^2(K)}^2 \right), \quad \forall w \in H^1(K).
    \label{eq:h1traceestimate}
  \end{equation}
  \label{le:h1traceestimate}
\end{lemma}
See the proof of this lemma in \cite{Wu2012unfitted,
  Hansbo2002unfittedFEM, Xiao2020high}. 
  
Now we are ready to claim the continuity and coercivity of the
bilinear form $b_h(\cdot, \cdot)$.
\begin{theorem}
  Let $b_h(\cdot, \cdot)$ be the bilinear form defined in
  \eqref{eq:bilinear} with sufficiently large $\eta$. Then there
  exist positive constants $C$ such that 
  \begin{align}
    |b_h(u, v) | &\leq C\enorm{u} \enorm{v}, \quad \forall u, v \in
    W_h, \label{eq:continuity}  \\
    b_h(v_h, v_h) &\geq C \enorm{v_h}^2, \quad \forall v_h \in V_h.
    \label{eq:coercivity}
  \end{align}
  \label{th:coercivity}
\end{theorem}
\begin{proof}
  By Cauchy-Schwartz inequality, for $\forall u, v \in W_h$ we
  directly obtain that 
  \begin{displaymath}
    \begin{aligned}
      b_h(u, v) &\leq C \Big( \|\beta \nabla u \|_{L^2(\MTh^0 \cup
      \MTh^1)}^2 + \| h_e^{-1/2} \jump{u} \|_{L^2(\MEh^0 \cup
      \MEh^1)}^2 + \| h_e^{1/2} \aver{\beta \nabla u} \|_{L^2(\MEh^0
      \cup \MEh^1)}^2 + \|h_K^{-1/2} \jump{u} \|_{L^2(\Gamma)}^2 \\ 
      +  \| & h_K^{1/2} \aver{\beta \nabla u} \|_{L^2(\Gamma)}^2
      \Big)^{1/2} \Big( \|\beta \nabla v \|_{L^2(\MTh^0 \cup
      \MTh^1)}^2  + \| h_e^{-1/2} \jump{v} \|_{L^2(\MEh^0 \cup
      \MEh^1)}^2 + \| h_e^{1/2} \aver{\beta \nabla v} \|_{L^2(\MEh^0
      \cup \MEh^1)}^2 \\
      &\hspace{180pt} + \|h_K^{-1/2} \jump{v} \|_{L^2(\Gamma)}^2 +
      \|h_K^{1/2} \aver{\beta \nabla v} \|_{L^2(\Gamma)}^2 \Big)^{1/2}
      \\
      & \leq C \enorm{u} \enorm{v}, \\
    \end{aligned}
  \end{displaymath}
  which directly gives us the continuity result \eqref{eq:continuity}.

  To obtain \eqref{eq:coercivity}, we first define a weaker norm
  $\wnorm{\cdot}$ which is a more natural one for analyzing
  coercivity. For any $w_h \in V_h$, $\wnorm{\cdot}$ is given by
  \begin{displaymath}
    \begin{aligned}
    \wnorm{w_h}^2 = \|\nabla w_h\|_{L^2(\MTh^0 \cup \MTh^1)}^2 +&
    \|h_e^{-1/2} \jump{w_h} \|_{L^2(\MEh^0 \cup \MEh^1)}^2   + \|
    h_K^{-1/2} \jump{w_h}\|_{L^2(\Gamma)}^2. \\
  \end{aligned}
  \end{displaymath}
  From the trace estimate \eqref{eq:traceinequality} and the inverse
  inequality \eqref{eq:inverseinequality}, we immediately obtain that
  \begin{displaymath}
    \begin{aligned}
      \sum_{e \in \partial K} \|h_e^{1/2} \nabla w_h \|_{L^2(e)}^2 &
      \leq C \sum_{e \in \partial K} \left( h_K^{-1} \| h_e^{1/2}
      \nabla w_h \|_{L^2(K)}^2 + h_K \|h_e^{1/2} \nabla^2
      w_h\|_{L^2(K)}^2  \right)   \\
      &\leq C \|\nabla w_h\|_{L^2(K)}^2, \quad \forall K \in \MThB. \\
    \end{aligned}
  \end{displaymath}
  \revise{
  By the trace estimate \eqref{eq:traceestimate} and the mesh
  regularity M1, we have that
  \begin{displaymath}
    \begin{aligned}
      \sum_{e \in \partial K} \|h_e^{1/2} \nabla w_h \|_{L^2(e^0 \cup
      e^1)}^2 \leq C \left( \|h_K^{1/2} \nabla w_h \|_{L^2(\partial
      K^0)}^2 +  \|h_K^{1/2} \nabla w_h \|_{L^2(\partial K^1)}^2
      \right), \quad \forall K \in \MThG, 
    \end{aligned}
  \end{displaymath}
  and
  \begin{displaymath}
    \begin{aligned}
      \| h_K^{1/2} \nabla w_h\|_{L^2(\partial K^i)} &\leq C h_K^{-1}
      \|h_K^{1/2} \nabla w_h\|_{L^2(K_\circ^i)}^2 \leq C \| \nabla w_h
      \|_{L^2(K_\circ^i)}^2, \quad \forall K \in \MThG, \quad i= 0, 1.
      \\
    \end{aligned}
  \end{displaymath}
  The above inequalities give us 
  \begin{displaymath}
    \begin{aligned}
    \|h_e^{1/2} \aver{ \nabla & w_h} \|_{L^2(\MEh^0 \cup \MEh^1)}^2  +
    \|h_K^{1/2} \aver{ \nabla w_h}\|_{L^2(\Gamma)}^2 \leq \\ C &
    \left( \sum_{K \in \MThB} \sum_{e \in \partial K} \|h_e^{1/2}
    \nabla w_h\|_{L^2(e)}^2 + \sum_{K \in \MThG} \|h_K^{1/2} \nabla
    w_h \|_{L^2(\partial K^0)}^2 + \sum_{K \in \MThG} \|h_K^{1/2}
    \nabla w_h \|_{L^2(\partial K^1)}^2  \right) \\
    \leq C & \left( \sum_{K \in \MThB} \| \nabla w_h \|_{L^2(K)}^2 +
    \sum_{K \in \MThG} \|\nabla w_h \|_{L^2(K_\circ^0)}^2 +  \sum_{K
    \in \MThG} \|\nabla w_h \|_{L^2(K_\circ^1)}^2  \right) \\ 
    \leq C & \|\nabla w_h\|_{L^2(\MTh^0 \cup \MTh^1)}^2, \\
  \end{aligned}
\end{displaymath}}
  which actually indicates $\enorm{w_h} \leq C \wnorm{w_h}$ and the
  equivalence of $\enorm{\cdot}$ and $\wnorm{\cdot}$ restricted on
  $V_h$.

  Then we consider to bound the trace terms in the bilinear form with
  respect to the norm $\wnorm{\cdot}$. For the face $e \in
  \MEh^\circ$, we let $e$ be shared by two neighbouring elements $K^-$
  and $K^+$. For any $e \in \MEh^\circ \cap \MEhB$, we apply the
  Cauchy-Schwartz inequality to get that 
  \begin{equation}
    \begin{aligned}
      -\int_{e} 2 \jump{v_h} \cdot & \aver{\beta \nabla v_h} \d{s}
      \geq - \int_{e} \frac{1}{h_e \varepsilon } \| \jump{v_h} \|^2
      \d{s} - \int_{e} h_e \varepsilon \| \aver{\beta \nabla v_h}\|^2
      \d{s}, \\
      & \geq -\frac{1}{\varepsilon} \|h_e^{-1/2} \jump{v_h}
      \|_{L^2(e)}^2 - \varepsilon \|h_e^{1/2}  \beta \nabla v_h
      \|_{L^2(e \cap \partial K^-)}^2 -  \varepsilon \|h_e^{1/2}
       \beta \nabla v_h \|_{L^2(e \cap \partial K^+)}^2 ,\\
    \end{aligned}
    \label{eq:etrace} 
  \end{equation}
  for any $\varepsilon > 0$. For any $e \in \MEh^\circ \cap \MEhG$ and
  $i = 0, 1$, we deduce that 
  \begin{equation}
    \begin{aligned}
      -\int_{e^i} 2 \jump{v_h} \cdot & \aver{\beta \nabla v_h} \d{s}
      \geq - \int_{e^i} \frac{1}{h_e \varepsilon } \| \jump{v_h} \|^2
      \d{s} - \int_{e^i} h_e \varepsilon \| \aver{\beta \nabla
      v_h}\|^2 \d{s}, \\
      & \geq -\frac{1}{\varepsilon} \|h_e^{-1/2} \jump{v_h}
      \|_{L^2(e^i)}^2 - \varepsilon \|h_e^{1/2} \beta \nabla v_h
      \|_{L^2(e^i \cap \partial K^-)}^2 -  \varepsilon \|h_e^{1/2}
      \beta \nabla v_h \|_{L^2(e^i \cap \partial K^+)}^2 ,\\
    \end{aligned}
    \label{eq:eitrace}
  \end{equation}
  By trace inequality \eqref{eq:traceestimate} and
  \eqref{eq:traceinequality}, for any $e \in \MEh^\circ$ we have 
  \begin{equation}
    \|h_e^{1/2}\beta \nabla v_h \|_{L^2(e^i \cap \partial K^{\pm})}
    \leq 
    \begin{cases}
      C\| \nabla v_h \|_{L^2(K^{\pm})}, \quad K^{\pm} \in \MThB, \\
      C\| \nabla v_h \|_{L^2( (K^{\pm})_\circ^i)}, \quad K^{\pm} \in
      \MThG,\\
    \end{cases} i = 0, 1.
    \label{eq:Kpmtrace}
  \end{equation}
  \revise{
  Together with \eqref{eq:etrace} - \eqref{eq:Kpmtrace}, we obtain
  that 
  \begin{displaymath}
    \begin{aligned}
      \sum_{e \in \MEh^\circ \cap \MEhB} - \int_{e} 2 \jump{v_h} \cdot
      \aver{\beta \nabla v_h} \d{s} \geq -\sum_{e \in \MEh^\circ \cap
      \MEhB} \frac{1}{\varepsilon} \|h_e^{-1/2} \jump{v_h}
      \|_{L^2(e)}^2 - C \varepsilon \sum_{K \in \MTh} \| \nabla v_h
      \|_{L^2(K^0 \cup K^1)}^2, 
    \end{aligned}
  \end{displaymath}
  and
  \begin{displaymath}
    \begin{aligned}
      - \sum_{e \in \MEh^\circ \cap \MEhG} \left( \int_{e^0} 2
      \jump{v_h} \cdot \aver{\beta \nabla v_h} \d{s} +  \int_{e^1} 2
      \jump{v_h} \cdot \aver{\beta \nabla v_h} \d{s} \right) &\geq -
      \sum_{e \in \MEh^\circ \cap \MEhG} \frac{1}{\varepsilon}
      \|h_e^{-1/2} \jump{v_h} \|_{L^2(e^0)}^2 \\ 
      - \sum_{e \in \MEh^\circ \cap \MEhG} \frac{1}{\varepsilon}
      \|h_e^{-1/2} & \jump{v_h} \|_{L^2(e^1)}^2 - C \varepsilon
      \sum_{K \in \MTh} \| \nabla v_h \|_{L^2(K^0 \cup K^1)}^2. \\
    \end{aligned}
  \end{displaymath}
  For any $e \in \MEh^b$, it is similar to derive that
  \begin{displaymath}
    \sum_{e \in \MEh^b} - \int_{e} 2 \jump{v_h} \cdot \aver{\beta
    \nabla v_h} \d{s} \geq -\sum_{e \in \MEh^b} \frac{1}{\varepsilon}
    \|h_e^{-1/2} \jump{v_h} \|_{L^2(e)}^2 - C \varepsilon \sum_{K \in
    \MTh} \| \nabla v_h \|_{L^2(K^0 \cup K^1)}^2. 
  \end{displaymath}
  } \revise{
  Further for any $K \in \MThG$, we again apply the trace estimate
  \eqref{eq:traceestimate} to obtain that 
  \begin{equation}
    \begin{aligned}
      -\int_{\Gamma_K}& 2 \jump{v_h} \cdot  \aver{\beta \nabla v_h}
      \d{s} \geq \int_{\Gamma_K} -\frac{1}{\varepsilon} \| h_K^{-1/2}
      \jump{v_h} \|^2 \d{s} - \int_{\Gamma_K} \varepsilon \|
      h_K^{-1/2} \aver{\beta \nabla v_h}\|^2 \d{s}, \\
      &  \geq -\frac{1}{\varepsilon} \|h_K^{-1/2} \jump{v_h}
      \|_{L^2(\Gamma_K)}^2 - \varepsilon \|h_K^{1/2} \beta \nabla v_h
      \|_{L^2(\Gamma_K \cap \partial K^0)}^2 -  \varepsilon
      \|h_K^{1/2}  \beta \nabla v_h \|_{L^2(\Gamma_K \cap \partial
      K^1)}^2 ,\\
      & \geq -\frac{1}{\varepsilon} \|h_K^{-1/2} \jump{v_h}
      \|_{L^2(\Gamma_K)}^2 - C \varepsilon \left( \|\nabla v_h
      \|_{L^2(K_\circ^0)}^2 + \| \nabla v_h \|_{L^2(K_\circ^1)}^2
      \right). \\
    \end{aligned}
    \label{eq:GammaKtrace}
  \end{equation}}
  The inequality \eqref{eq:GammaKtrace} yields that 
  \begin{displaymath}
    \begin{aligned}
      - \sum_{K \in \MThG} \int_{\Gamma_K} 2 \jump{v_h} \cdot
      \aver{\beta \nabla v_h} \d{s} \geq - \sum_{K \in \MThG}
      \frac{1}{\varepsilon} \|h_K^{-1/2} \jump{v_h}
      \|_{L^2(\Gamma_K)}^2 - C \varepsilon \sum_{K
      \in \MTh} \| \nabla v_h \|_{L^2(K^0 \cup K^1)}^2.
    \end{aligned}
  \end{displaymath}
  Combining all above inequalities, we conclude that there exists
  a constant $C$ such that 
  \begin{displaymath}
    \begin{aligned}
    b_h(v_h, v_h) \geq (\beta - C \varepsilon) \|\nabla
    v_h&\|_{L^2(\MTh^0 \cup \MTh^1)}^2 + \|(\eta -
    \frac{1}{\varepsilon}) h_e^{-1/2} \jump{v_h}\|_{L^2(\MEh^0 \cup
    \MEh^1)}^2 \\ 
    & + \|(\eta - \frac{1}{\varepsilon}) h_K^{-1/2}
    \jump{v_h}\|_{L^2(\Gamma)}^2, \\
  \end{aligned}
  \end{displaymath}
  for any $\varepsilon > 0$. We can directly let $\varepsilon = \beta
  /(2C)$ and select a sufficiently large $\eta$ to ensure $b_h(v_h,
  v_h) \geq C \wnorm{v_h}^2$, which completes the proof.
\end{proof}


\begin{remark}
  To ensure the stability near the interface, some unfitted methods
  \cite{Xiao2020high, Massjung2012unfitted,
  Hansbo2002unfittedFEM} may require a weighted average $\aver{v}
  =\kappa_0 v|_{\Omega_0} + \kappa_1 v|_{\Omega_1}$ where $\kappa_0$
  and $\kappa_1$ are the cut-dependent parameters like $\kappa_i =
  |K_i|/|K|(i = 0, 1)$ for elements in $\MThG$. In our method, another
  advantage is just taking the arithmetic one could also guarantee the
  stability and we note that this advantage is brought by the patch
  reconstruction. In addition, the analysis can be adapted to their
  choices without any difficulty. 
\end{remark}

Now let us give the approximation error in the DG energy norm
$\enorm{\cdot}$.
\begin{lemma}
  Let $u \in H^{t}(\Omega_0 \cup \Omega_1)$ with $t \geq 2$, there
  exists a constant $C$ such that
  \begin{equation}
    \enorm{u - \mc R u} \leq C\Lambda_m h^{s - 1} \|u\|_{H^t(\Omega_0
    \cup \Omega_1)},
    \label{eq:appDGnorm}
  \end{equation}
  where $s = \min(m + 1, t)$.
  \label{th:appDGnorm}
\end{lemma}
\begin{proof}
  From \eqref{eq:interpolation}, it is trivial to obtain
  \begin{displaymath}
  \| \nabla( u - \mc R u)\|_{L^2(\MTh^0 \cup \MTh^1)} \leq
  C\Lambda_m h^{s - 1} \|u\|_{H^t(\Omega_0 \cup \Omega_1)}.
  \end{displaymath}
  Then using trace inequality \eqref{eq:traceinequality} and
  \eqref{eq:interpolation}, for any $K \in
  \MTh^i(i=0,1)$ we have
  \begin{displaymath}
    \begin{aligned}
      \|u - \mc R u\|_{L^2\left( (\partial K)^i \right)} &\leq \|E^i u
      - \mc R(E^i u)\|_{L^2(\partial K)} \leq C\Lambda_m h_K^{s+1/2}
      \|E^i u\|_{H^t(S^i(K))} ,\\
      \|\nabla(u - \mc R u)\|_{L^2\left( (\partial K)^i \right)} &\leq
      \|\nabla(E^i u - \mc R(E^i u))\|_{L^2(\partial K)} \leq
      C\Lambda_m h_K^{s-1/2} \|E^i u\|_{H^t(S^i(K))}. \\
    \end{aligned}
  \end{displaymath}
  From the above two inequalities and \eqref{eq:extension}, we could
  conclude 
  \begin{displaymath}
    \begin{aligned}
      \|h_e^{-1/2} \jump{u - \mc R u}\|_{L^2(\MEhB)} +   \|h_e^{-1/2}
      \jump{u - \mc R u}\|_{L^2(\MEhG)}  & \leq C \Lambda_m h^{s - 1}
      \|u\|_{H^t(\Omega_0 \cup \Omega_1)}, \\
     \|h_e^{1/2} \aver{u - \mc Ru}\|_{L^2(\MEhB)} + \|h_e^{1/2}
     \aver{u - \mc Ru}\|_{L^2(\MEhG)}&\leq C \Lambda_m h^{s -
     1}\|u\|_{H^t(\Omega_0 \cup \Omega_1)}. \\
    \end{aligned}
  \end{displaymath}
  Finally we use \eqref{eq:h1traceestimate} to  bound the error on the
  interface. For any $K \in \MThG$, we obtain 
  \begin{displaymath}
    \begin{aligned}
      \|h_K^{-1/2} \jump{u - \mc R u} \|_{L^2(\Gamma_K)} &\leq C
      \sum_{i = 0, 1} \big ( h_K^{-1} \|E^iu - \mc R(E^i u)\|_{L^2(K)}
      \\
      & \hspace{100pt}+ h_K \|\nabla(E^i u - \mc R(E^i u)) \|_{L^2(K)}
      \big ) \\
      &\leq C\Lambda_m h_K^{s-1} \left( \|E^0 u\|_{H^t(S^0(K))} +
      \|E^1u\|_{H^t(S^1(K))} \right). \\
    \end{aligned}
  \end{displaymath}
  A summation over all $K \in \MThG$ gives us 
  \begin{displaymath}
    \|h_K^{-1/2}\jump{u - \mc Ru}\|_{L^2(\Gamma)} \leq C\Lambda_m h^{s -
    1} \|u\|_{H^t(\Omega_0 \cup \Omega_1)}.
  \end{displaymath}
  Similarly, we could yield
  \begin{displaymath}
    \|h_K^{1/2}\aver{u - \mc R u}\|_{L^2(\Gamma)} \leq C \Lambda_m
    h^{s - 1} \|u\|_{H^t(\Omega_0 \cup \Omega_1)}.
  \end{displaymath}
  Combining all the inequalities above gives the error estimate
  \eqref{eq:appDGnorm}, which completes the proof.
\end{proof}

We are now ready to prove a priori error estimates.
\begin{theorem}
  Let $u \in H^t(\Omega_0 \cup \Omega_1)$ with $t \geq 2$ be the exact
  solution to \eqref{eq:elliptic} and let $u_h \in V_h$ be the
  solution to \eqref{eq:discrete}, then there exist constants $C$ such
  that the following error estimates hold true:
  \begin{equation}
    \enorm{u - u_h} \leq Ch^{s - 1} \|u\|_{H^t(\Omega_0 \cup
    \Omega_1)},
    \label{eq:estimateDGnorm}
  \end{equation}
  and
  \begin{equation}
    \|u - u_h\|_{L^2(\Omega)} \leq Ch^{s} \|u\|_{H^t(\Omega_0 \cup
    \Omega_1)},
    \label{eq:estimateL2norm}
  \end{equation}
  where $s = \min(m+1, t)$.
  \label{th:estimate}
\end{theorem}
\begin{proof}
 Together with the Galerkin orthogonality \eqref{eq:orthogonality},
 boundedness \eqref{eq:continuity} and coercivity
 \eqref{eq:coercivity} of the bilinear form $b_h(\cdot, \cdot)$ we
 could have a bound of $\enorm{u - u_h}$. For any $v_h \in V_h$, we
 obtain that 
 \begin{displaymath}
   \begin{aligned}
     C_0 \enorm{u_h - v_h}^2 &\leq b_h(u_h - v_h, u_h - v_h) = b_h(u -
     v_h, u_h - v_h) \\
     & \leq C_1 \enorm{u - v_h} \enorm{u_h - v_h}. \\
    \end{aligned}
 \end{displaymath}
 Hence, 
 \begin{displaymath}
   \begin{aligned}
     \enorm{u - u_h} &\leq \enorm{u - v_h} + \enorm{u_h - v_h} \leq C
     \enorm{u - v_h}  \\
     & \leq C \inf_{v_h \in V_h} \enorm{u - v_h} \leq C \enorm{u -
     \mc{R} u}. \\
   \end{aligned}
 \end{displaymath}
 Combining \eqref{eq:appDGnorm} immediately gives us the estimate
 \eqref{eq:estimateDGnorm}.
 
 Finally we obtain the \revise{optimal convergence order} in $L^2$
 norm with the standard duality argument. Let $\phi \in H^2(\Omega_0
 \cup \Omega_1)$ be the solution of
 \begin{displaymath}
   \begin{aligned}
     -\nabla \cdot \beta \nabla \phi &= u - u_h ,\quad \bm x \in
    \Omega_0 \cup \Omega_1, \\
    \phi  &= 0, \quad \bm x \in \partial \Omega, \\
    \jump{\phi} &= 0, \quad \bm x \in \Gamma, \\
    \jump{\beta \nabla \phi \cdot \bm{\mr n}_\Gamma} &= 0, \quad
    \bm x \in \Gamma, \\
  \end{aligned}
\end{displaymath}
and satisfy \cite{Babuska1970discontinuous}
\begin{displaymath}
  \|\phi\|_{H^2(\Omega_0 \cup \Omega_1)} \leq C \|u -
  u_h\|_{L^2(\Omega)}.
\end{displaymath}
We denote by $\phi_I = \mc R \phi$ the interpolant of $\phi$. Then
together with the Galerkin orthogonality \eqref{eq:orthogonality} we
deduce that
\begin{displaymath}
  \begin{aligned}
    \|u - u_h\|_{L^2(\Omega)}^2 &= b_h(\phi, u - u_h) = b_h(\phi -
    \phi_I, u - u_h) \\
    &\leq C \enorm{\phi - \phi_I} \enorm{u - u_h} \leq Ch
    \|\phi\|_{H^2(\Omega_0 \cup \Omega_1)} \|u\|_{H^t(\Omega_0 \cup
    \Omega_1)} \\
    & \leq C h^s \|u - u_h\|_{L^2(\Omega)} \|u\|_{H^t(\Omega_0 \cup
    \Omega_1)}.\\
  \end{aligned}
\end{displaymath}
The estimate \eqref{eq:estimateL2norm} is obtained by elminating $\| u
- u_h\|_{L^2(\Omega)}$, which completes the proof.
\end{proof}



\section{Numerical Experiments}
\label{sec:numericalresults}
In this section, we present some numerical results by solving some
benchmark elliptic interface problems. For each case, the source term
$f$, the Dirichlet boundary data $g$ and the jump term $a$, $b$ are
given according to the solutions. We construct the spaces of order
$1 \leq m \leq 3$ to solve each problem. For simplicity, we take the
$\# S^i(K)$ uniformly for all elements and we list the values of $\#
S^i(K)$ for all $m$ that are used in all experiments in Tab
\ref{tab:numSK}. A direct sparse solver is used to solve the resulting
sparse linear system. The interface in all numerical experiments is
described by a given level set function $\phi(\bm{x})$.

\begin{table}
  \centering
  \renewcommand\arraystretch{1.3}
  \begin{tabular}{p{1.5cm}|p{1.5cm}|p{1cm} |p{1cm} |p{1cm} }
    \hline\hline
    & $m$ & 1 & 2 & 3 \\
    \hline
    \multirow{2}{1.5cm}{$\# S^i(K)$} & $d=2$ & 5 & 9 & 15 \\
    \cline{2-5}
    & $d=3$ & 9 & 18 & 38 \\
    \hline\hline
  \end{tabular}
  \caption{ The uniform $\# S^i (K)$ for $ 1 \leq m \leq 3$.}
  \label{tab:numSK}
\end{table}

\subsection{2D Example}
\paragraph{\textbf{Example 1.}} We first consider the classical
interface problem on the square domain $(-1, 1) \times (-1, 1)$ with a
circular interface $\phi(x, y) = x^2 + y^2 - r^2$ with radius $r =
0.5$ (see Fig \ref{fig:mesh_circle}). The exact solution and
coefficient are chosen to be 
\begin{displaymath}
  \begin{aligned}
  u(x, y) & = \begin{cases}
    \frac14\left( 1 - \frac1{8b} - \frac1b \right) + \frac1b \left(
    \frac{r^4}{2} + r^2 \right), \quad & \text{outside $\Gamma$},\\
    x^2 + y^2, \quad & \text{inside $\Gamma$}, \\
  \end{cases} \\
  \beta &= \begin{cases}
    b, \quad & \text{outside $\Gamma$}, \\
    2, \quad & \text{inside $\Gamma$}. \\
  \end{cases} \\
\end{aligned}
\end{displaymath}
With $b = 10$, $u$ is continuous over $\Omega$.
\begin{figure}[htp]
  \centering
  \includegraphics[width=2.2in]{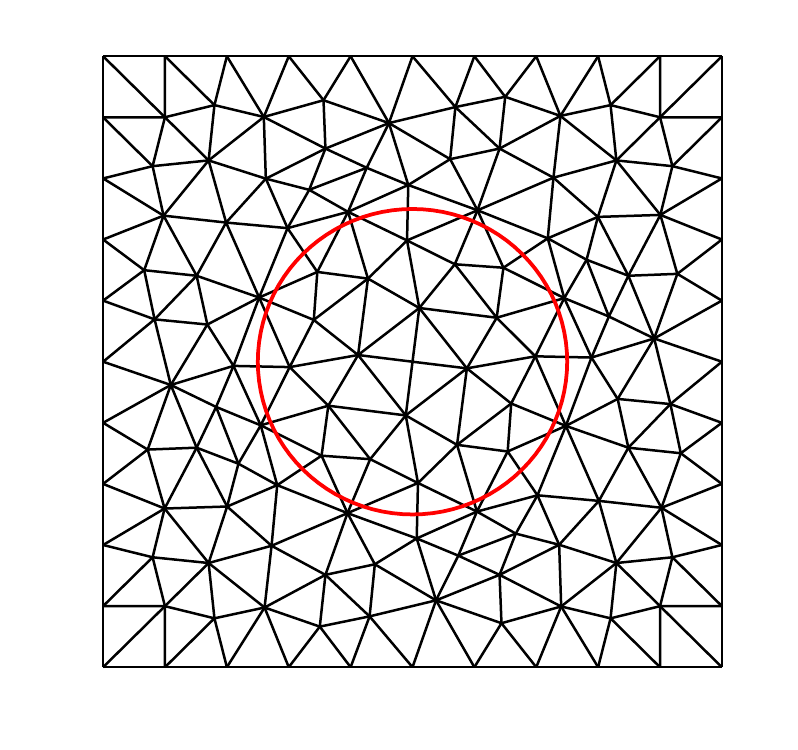}
  \includegraphics[width=2.2in]{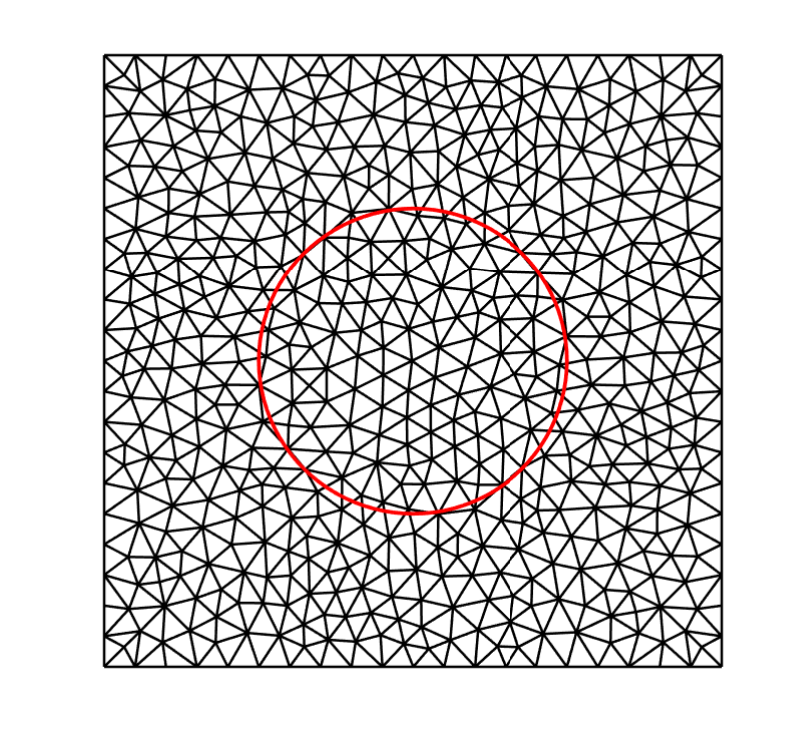}
  \caption{Triangulation for example 1 with mesh size $h = 1/5$ (left)
    / $h = 1/10$ (right). }
  \label{fig:mesh_circle}
\end{figure}
By using a series of quasi-uniform triangular meshes, the $L^2$ norm
and DG energy norm of the error in the approximation to the exact
solution with mesh size $h = 1/5, 1/10, \ldots , 1/80$ are reported in
Fig \ref{fig:ex1}. For each fixed $m$, we observe that the errors $\|u
- u_h\|_{L^2(\Omega)}$ and $\enorm{u - u_h}$ converge to zero at the
rate $O(h^{m+1})$ and $O(h^m)$ as the mesh is refined, respectively.
Such convergence rates are consistent with the theoretical results.

\begin{figure}[htb]
  \centering
  \includegraphics[width=2.2in]{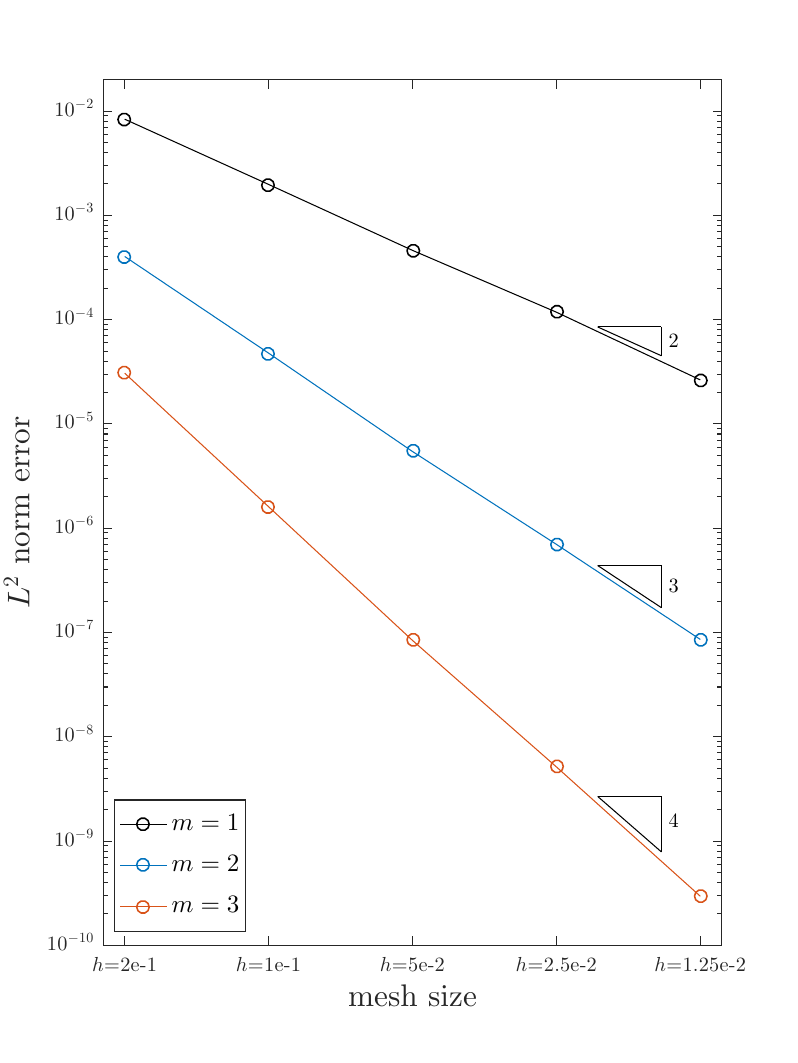}
  \includegraphics[width=2.2in]{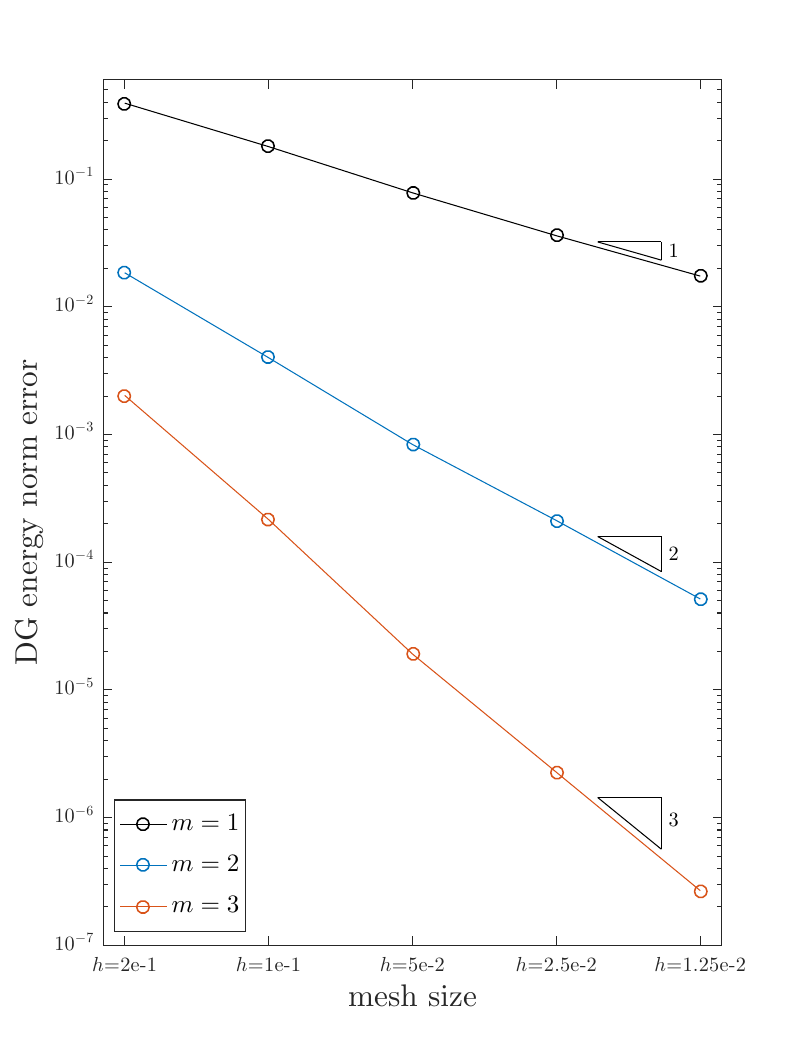}
  \caption{The convergence orders under $L^2$ norm (left) / DG
    energy norm (right) for Example 1.}
  \label{fig:ex1}
\end{figure}

\paragraph{\textbf{Example 2.}} In this example, we consider the same
interface and the same domain as in Example 1. The analytical solution
$u(x, y)$ and the coefficient are defined in the same way as in
Example 1. But we solve the elliptic interface problem based on a
sequence of polygonal meshes as shown in Fig \ref{fig:voronoi}, which
are generated by {\tt PolyMesher} \cite{talischi2012polymesher}.
\begin{figure}
  \centering
  \includegraphics[width=2.2in]{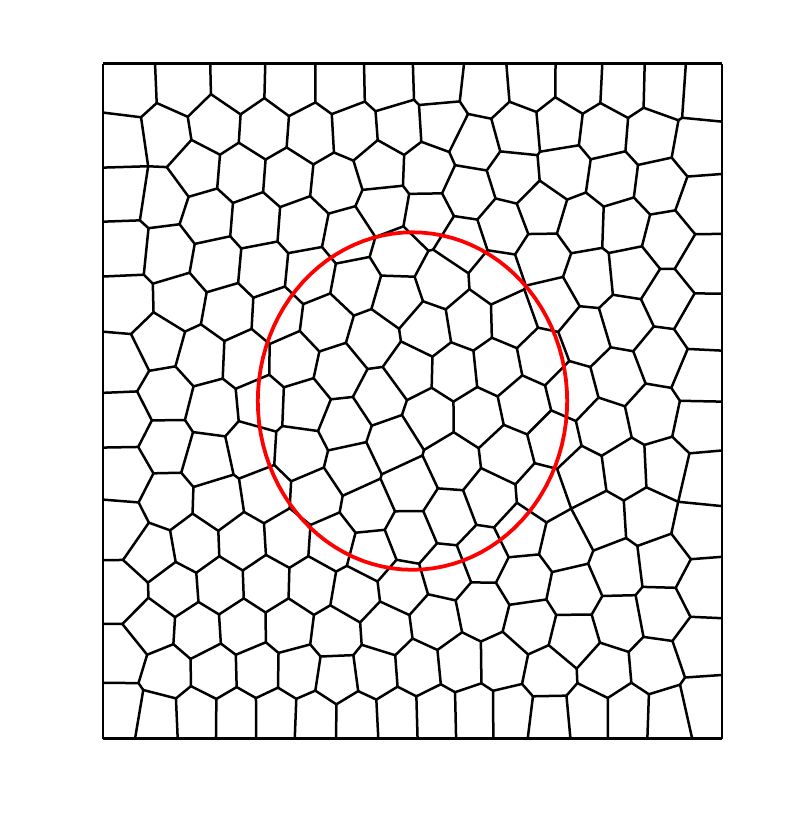}
  \includegraphics[width=2.2in]{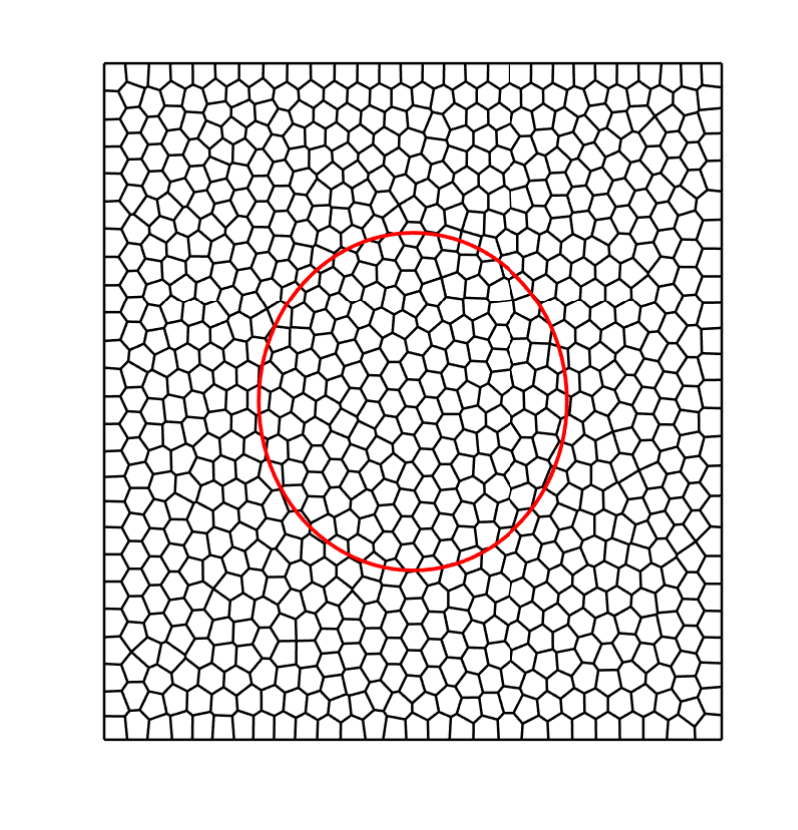}
  \caption{Voronoi mesh for example 2 with 200 elements (left) / 800
    elements (right). }
  \label{fig:voronoi}
\end{figure}
The numerically detected convergence orders are displayed in Fig
\ref{fig:ex2} for both error measurements. It is clear that the orders
of convergence in $L^2$ norm and DG energy norm are $O(h^{m+1})$ and
$O(h^m)$, respectively, which again are in agreement with the
theoretical predicts.

\begin{figure}[htb]
  \centering
  \includegraphics[width=2.2in]{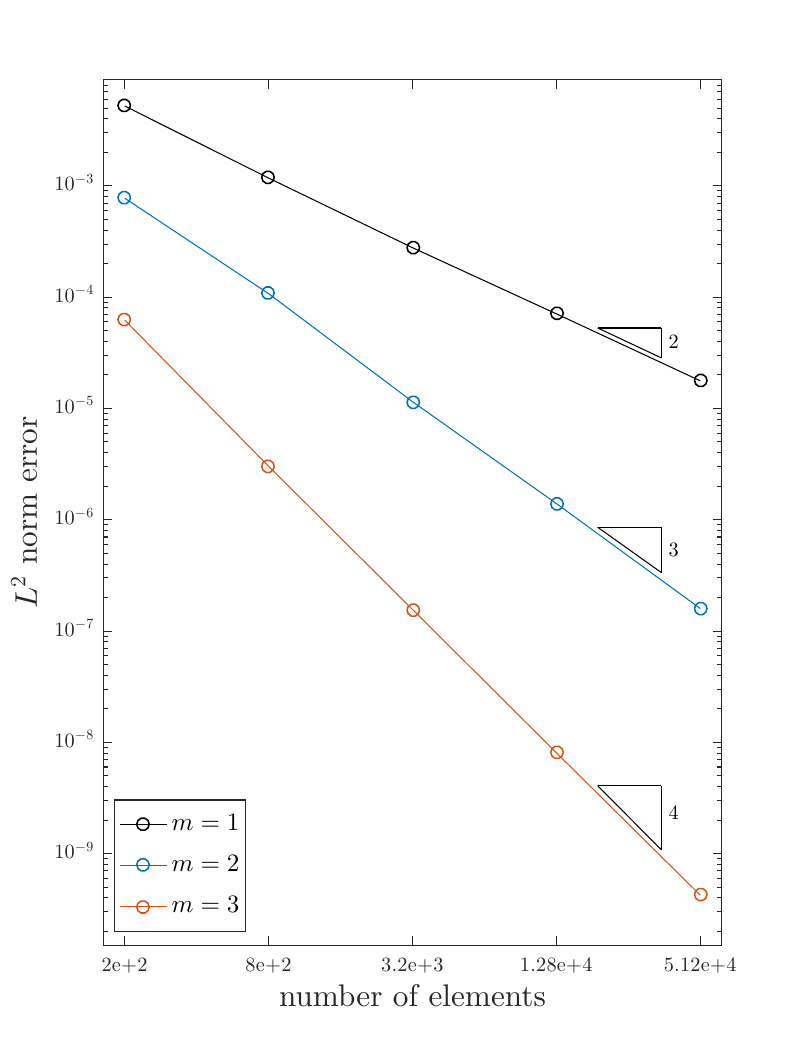}
  \includegraphics[width=2.2in]{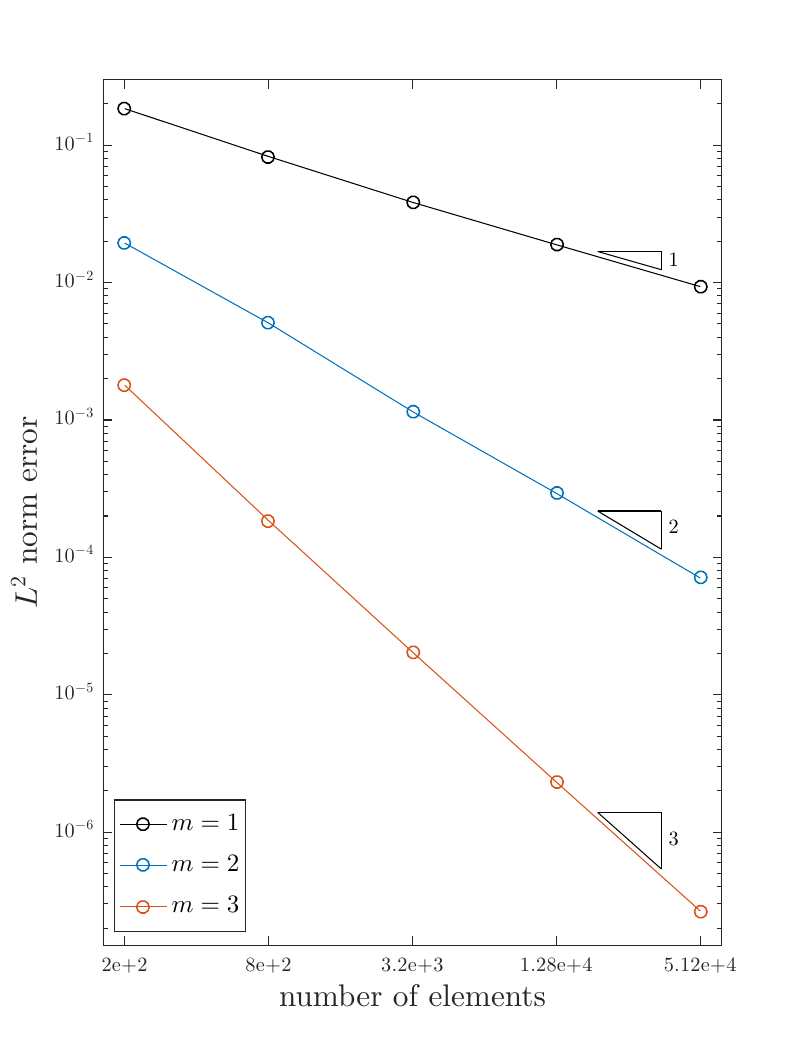}
  \caption{The convergence orders under $L^2$ norm (left) / DG
    energy norm (right) for Example 2.}
  \label{fig:ex2}
\end{figure}

For the Example 3 - 6, the computational domain is $(-1 , 1) \times
(-1, 1)$ and we solve the test problems on a sequence of triangular
meshes with mesh size $h = 1/5, 1/10, \cdots, 1/80$.

\paragraph{\textbf{Example 3.}} In this case, we consider the problem
in \cite{Zhou2006fictitious} which contains the strongly discontinuous
coefficient $\beta$ to test the robustness of the proposed method. We
consider the elliptic problem with an ellipse interface (see Fig
\ref{fig:mesh_ex3}), 
\begin{displaymath}
  \phi(x, y) = \left( \frac{x}{18/27} \right)^2 + \left(
  \frac{y}{10/27} \right)^2 - 1
\end{displaymath}
The exact solution and the coefficient are given as
\begin{displaymath}
  \begin{aligned}
    u(x, y) &= \begin{cases}
      5e^{-x^2 - y^2}, \quad &\text{outside } \Gamma, \\
      e^x \cos(y), \quad &\text{inside } \Gamma,\\
    \end{cases} \\
    \beta &= \begin{cases}
      1, \quad & \text{outside } \Gamma, \\
      1000, \quad & \text{inside } \Gamma.\\
    \end{cases}\\
  \end{aligned}
\end{displaymath}
\begin{figure}[htp]
  \centering
  \includegraphics[width=2.2in]{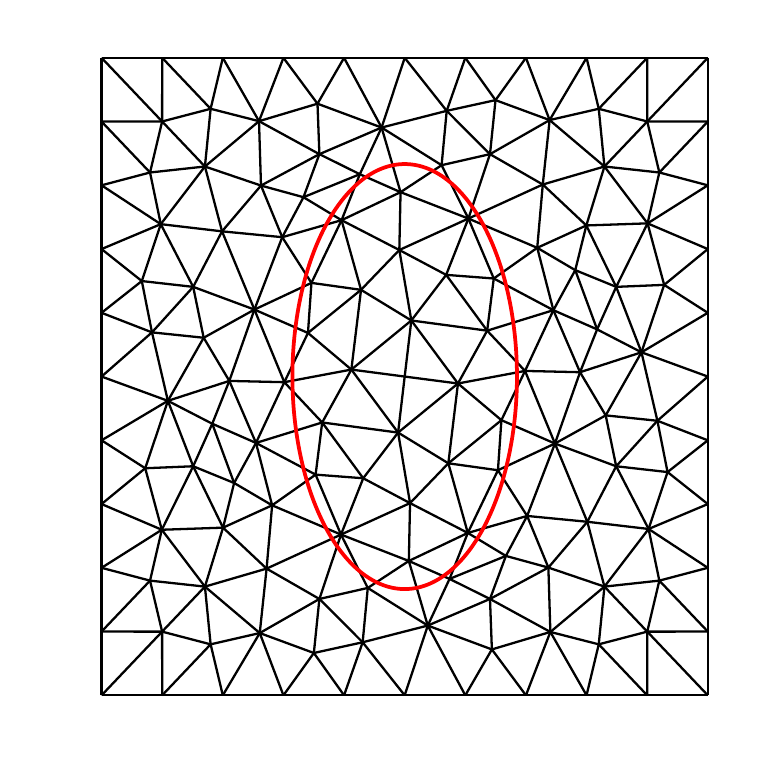}
  \includegraphics[width=2.2in]{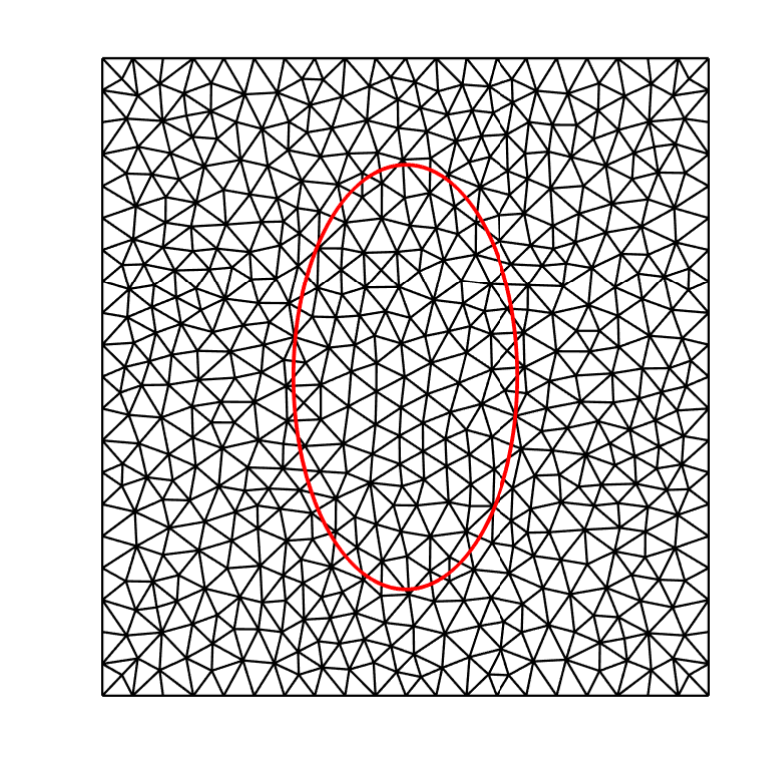}
  \caption{Triangulation for example 3 with mesh size $h = 1/5$ (left)
    / $h = 1/10$ (right).}
  \label{fig:mesh_ex3}
\end{figure}
There is a large jump in $\beta$ across the interface $\Gamma$, which
may lead to an ill-conditioned linear system. We still use the direct
sparse solver to solve the resulting sparse linear system and our
method shows the robustness for this case. As can be seen from
Fig \ref{fig:error_ex3}, the computed rates of convergence match with
the theoretical analysis.

\begin{figure}[htb]
  \centering
  \includegraphics[width=2.2in]{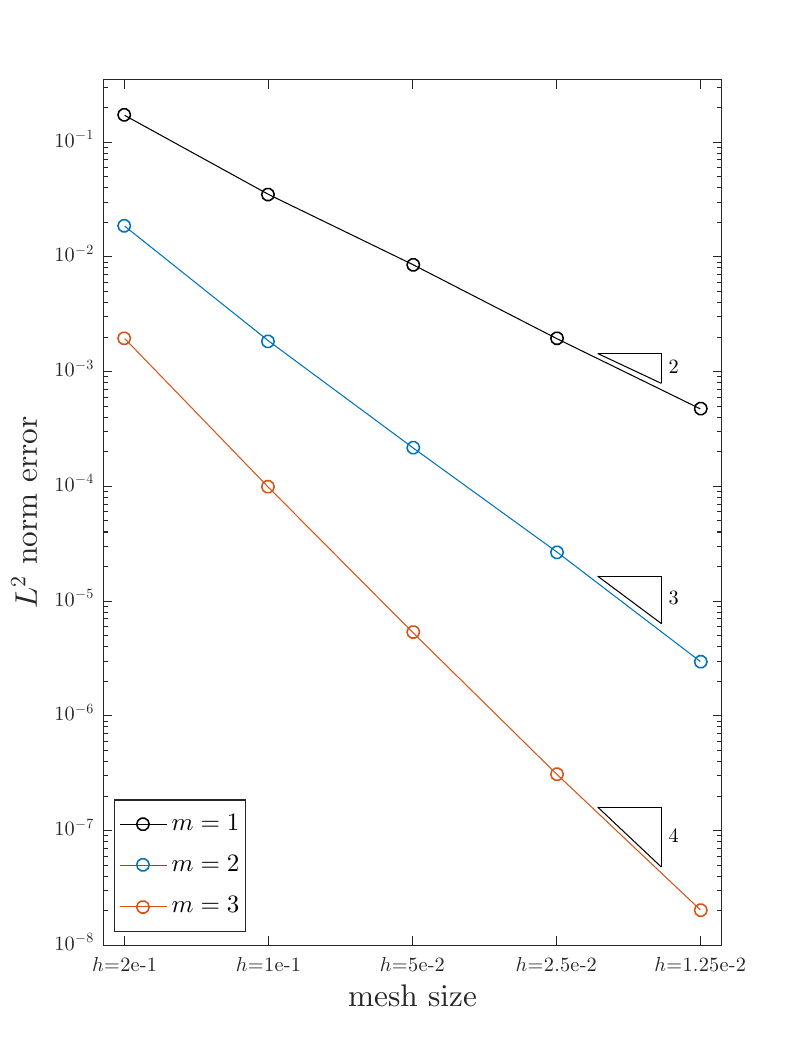}
  \includegraphics[width=2.2in]{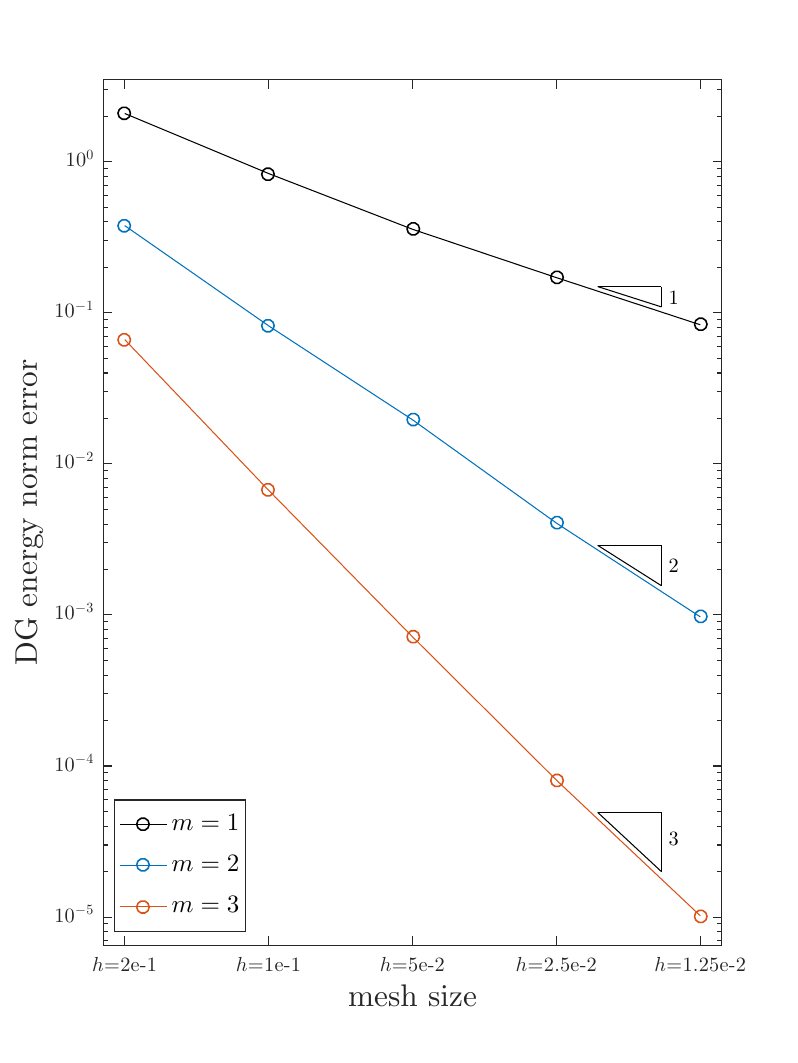}
  \caption{The convergence orders under $L^2$ norm (left) /  DG
    energy norm (right) for Example 3.}
  \label{fig:error_ex3}
\end{figure}

\paragraph{\textbf{Example 4.}} In this example, we consider solving
the elliptic problem with a kidney-shaped interface
\cite{Huynh2013hybrid}, which is governed by the following level set
function
\begin{displaymath}
  \phi(x, y) = \left( 2\left( \left( x + 0.5 \right)^2 + y \right) - x
  - 0.5 \right)^2 - \left( \left( x + 0.5 \right)^2 + y^2 \right) +
  0.1.
\end{displaymath}
\begin{figure}[htp]
  \centering
  \includegraphics[width=2.2in]{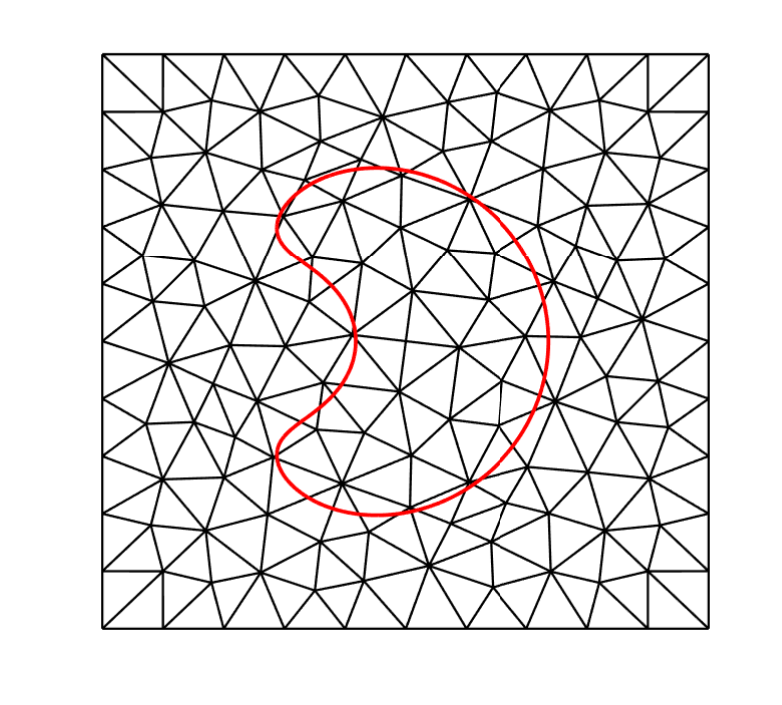}
  \includegraphics[width=2.2in]{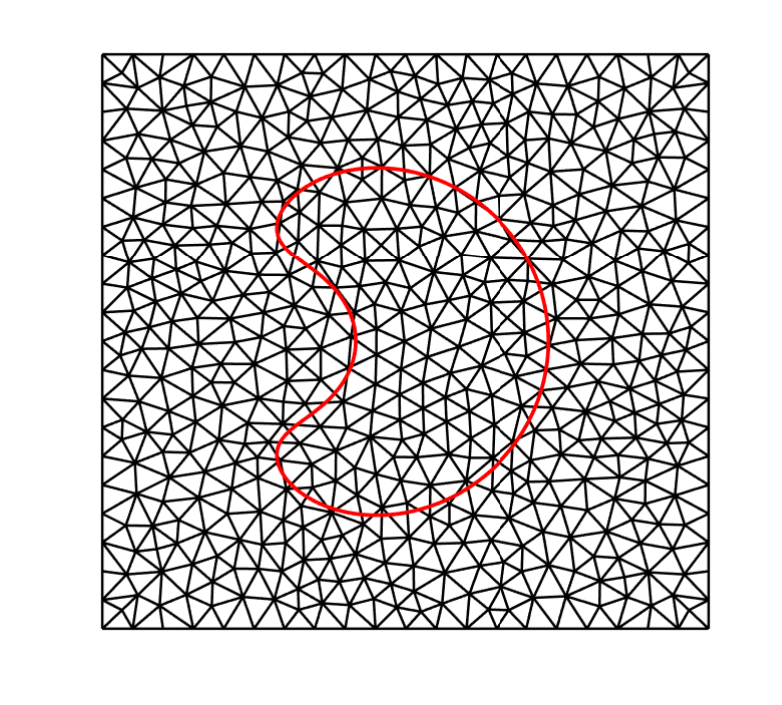}
  \caption{Triangulation for example 4 with mesh size $h = 1/5$ (left)
    / $h = 1/10$ (right).}
  \label{fig:mesh_ex4}
\end{figure}
The boundary data and source term are derived from the exact solution
and coefficient
\begin{displaymath}
  \begin{aligned}
    u(x, y) &= \begin{cases}
      0.1\cos(1 - x^2 - y^2), &\quad \text{outside } \Gamma, \\
      \sin(2x^2 + y^2 + 2) + x, &\quad \text{inside } \Gamma, \\
    \end{cases}\\
    \beta &= \begin{cases}
      10, &\quad \text{outside } \Gamma, \\
      1, &\quad \text{inside } \Gamma.\\
    \end{cases}\\
  \end{aligned}
\end{displaymath}
We present numerical results in Fig \ref{fig:error_ex4} and the
predicted convergence rates for both norms are verified.

\begin{figure}[htb]
  \centering
  \includegraphics[width=2.2in]{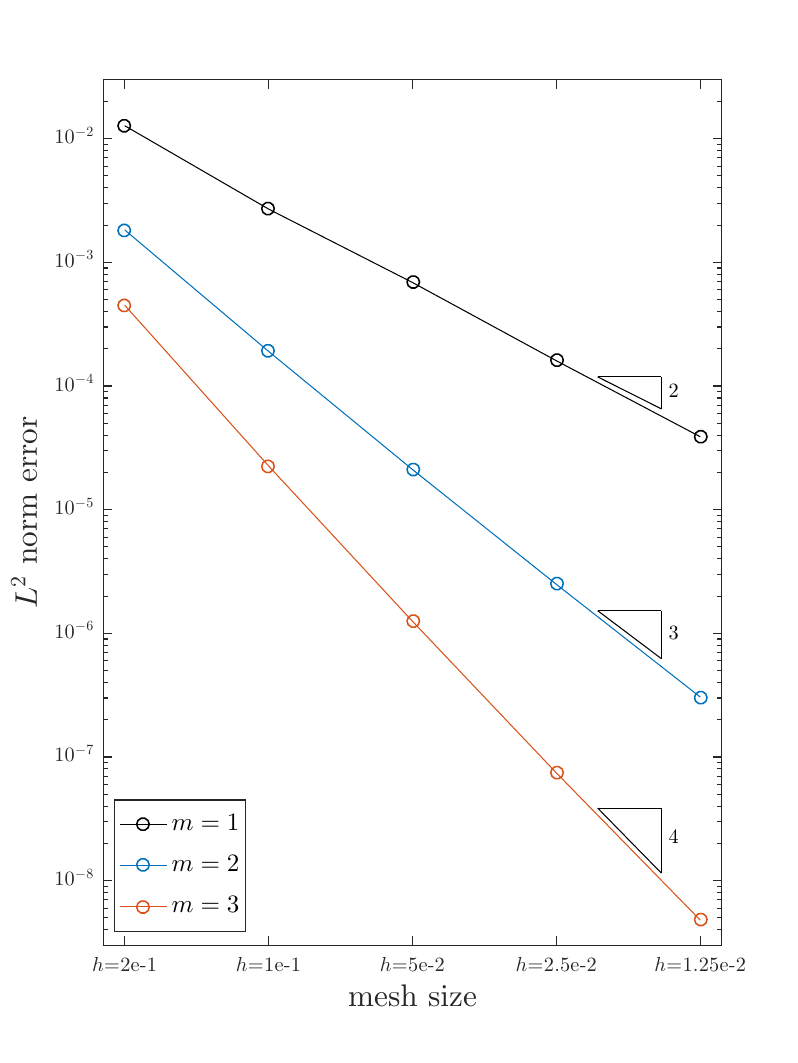}
  \includegraphics[width=2.2in]{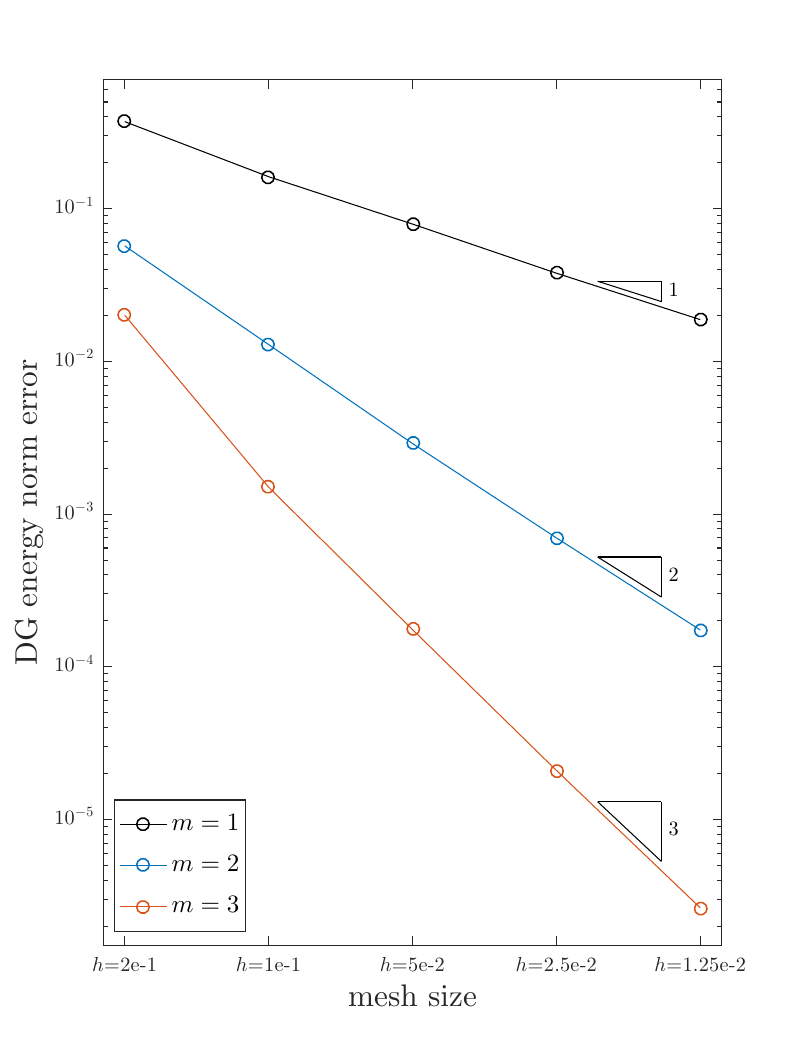}
  \caption{The convergence orders under $L^2$ norm (left) / DG
    energy norm (right) for Example 4.}
  \label{fig:error_ex4}
\end{figure}

\paragraph{\textbf{Example 5.}} Next, we consider a standard test case
with an interface consisting of both concave and convex curve segments
\cite{Zhou2006fictitious}. The interface is parametrized with the
polar angle $\theta$
\begin{displaymath}
  r = \frac{1}{2} + \frac{\sin 5\theta}{7}.
\end{displaymath}
The exact solution is selected to be 
\begin{displaymath}
  \begin{aligned}
    u(x, y) &= \begin{cases}
      0.1(x^2 + y^2)^2 - 0.01 \ln(2\sqrt{x^2 + y^2}), \quad
      &\text{outside } \Gamma,\\
      e^{x^2 + y^2}, \quad &\text{inside } \Gamma, \\
    \end{cases}\\
    \beta &= \begin{cases}
      10, \quad & \text{outside } \Gamma,\\
      1, \quad &\text{inside } \Gamma. \\
    \end{cases} \\
  \end{aligned}
\end{displaymath}
\begin{figure}[htb]
  \centering
  \includegraphics[width=2.2in]{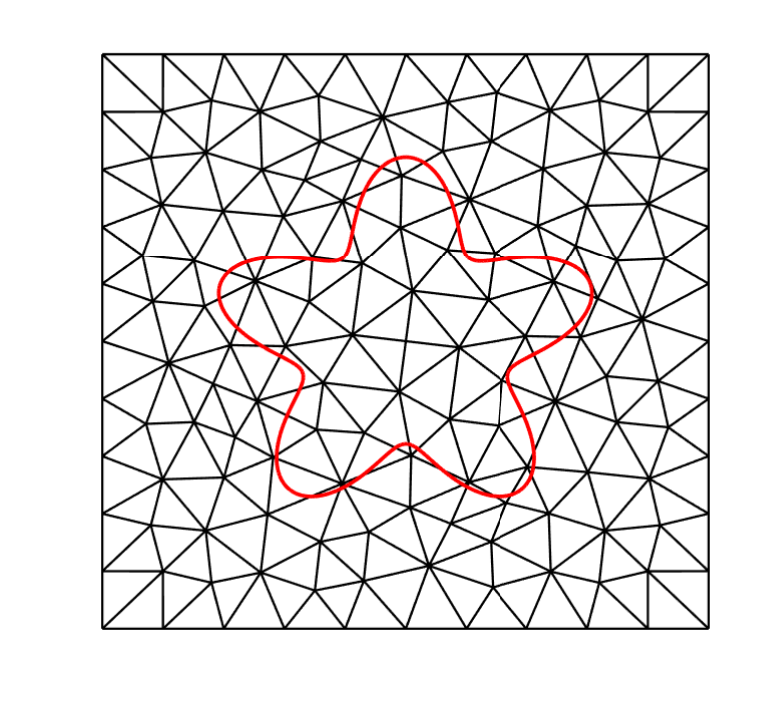}
  \includegraphics[width=2.2in]{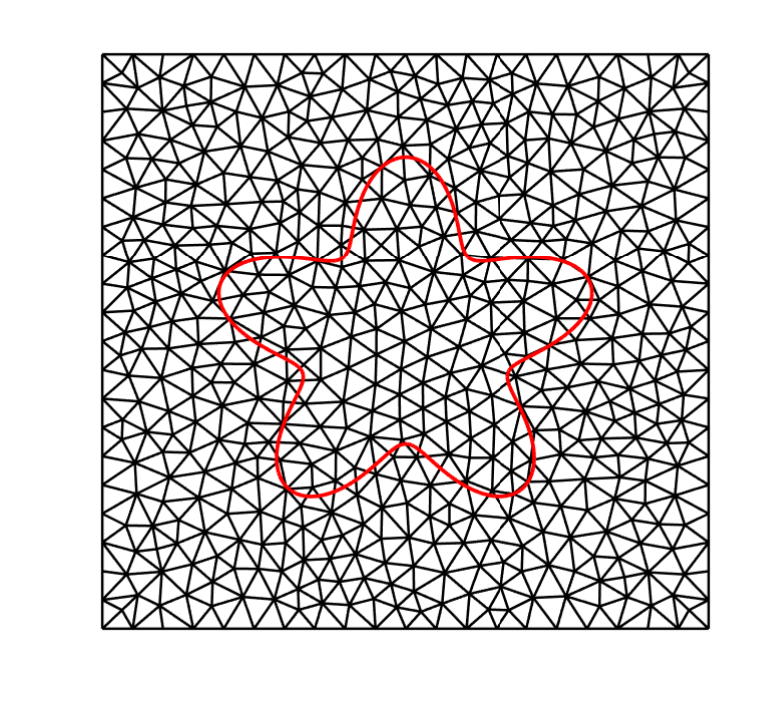}
  \caption{Triangulation for Example 6 with mesh size $h = 1/5$ (left)
    / $h = 1/10$ (right).}
  \label{fig:ex5}
\end{figure}
The convergence of the numerical solutions is displayed in Fig
\ref{fig:error_ex5}. Again we observe optimal rates of convergence for
both norms as the mesh size is decreased.

\begin{figure}[htb]
  \centering
  \includegraphics[width=2.2in]{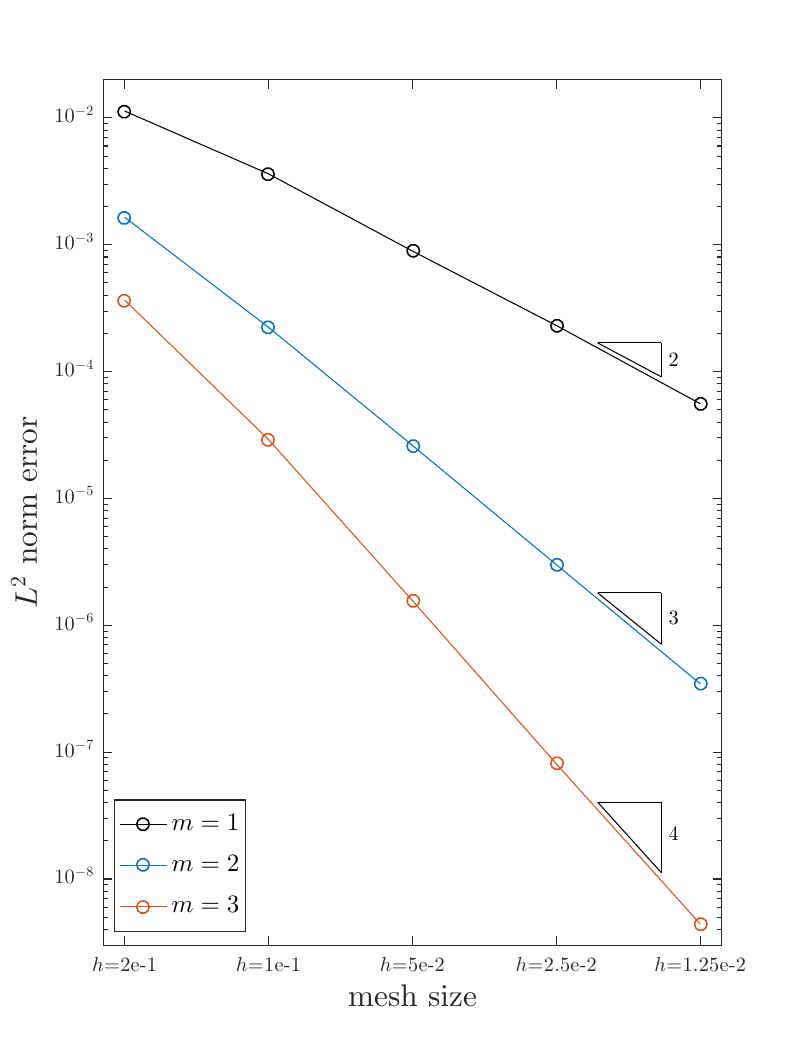}
  \includegraphics[width=2.2in]{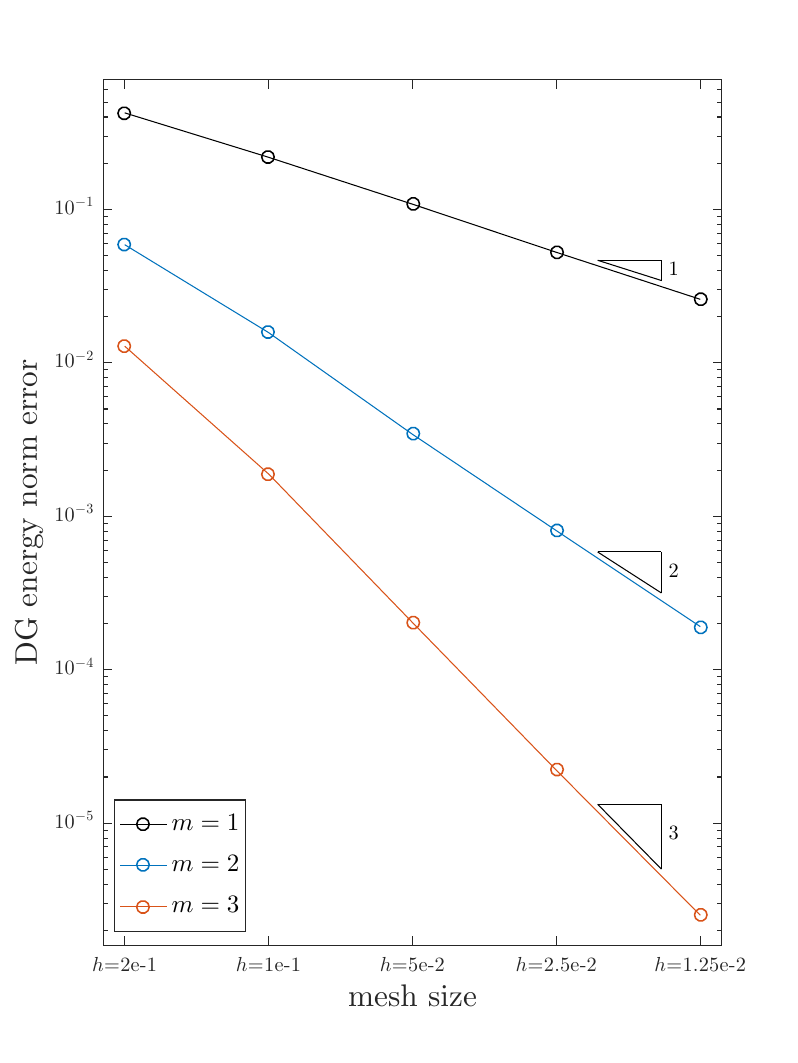}
  \caption{The convergence orders under $L^2$ norm (left) / DG
    energy norm (right) for Example 5.}
  \label{fig:error_ex5}
\end{figure}

\paragraph{\textbf{Example 6.}} In this case, we investigate the
performance of our proposed method when dealing with the problem with
low regularities. The interface can be found in
\cite{Hou2010numerical}, which is governed by the following level set
function
\begin{displaymath}
  \phi(x, y) = \begin{cases}
    y - 2x, &\quad x + y > 0, \\
    y + 0.5x, & \quad x + y \leq 0.\\
  \end{cases}
\end{displaymath}
\begin{figure}[htb]
  \centering
  \includegraphics[width=2.2in]{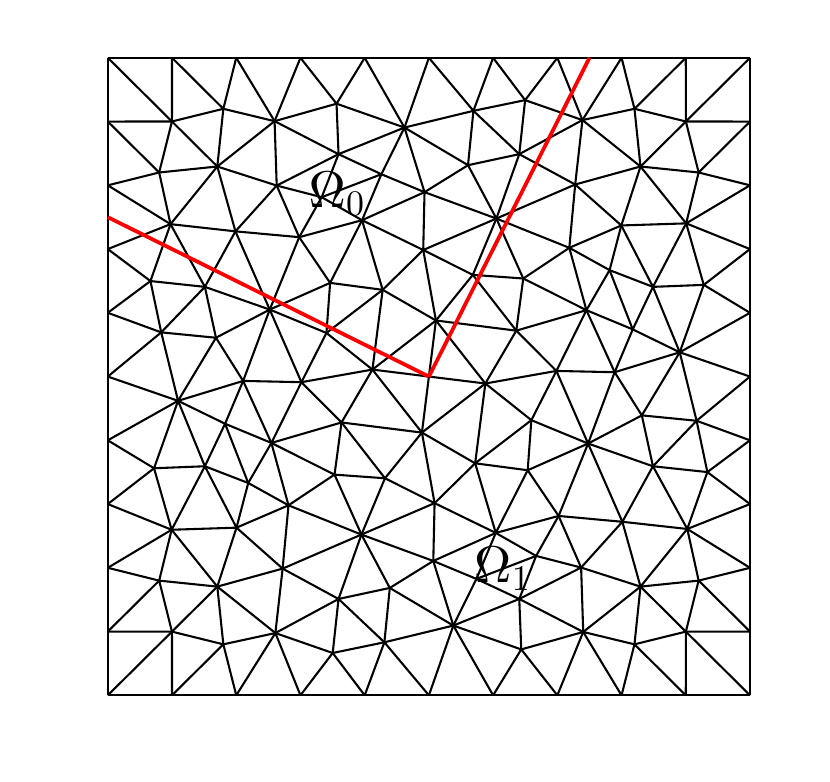}
  \includegraphics[width=2.2in]{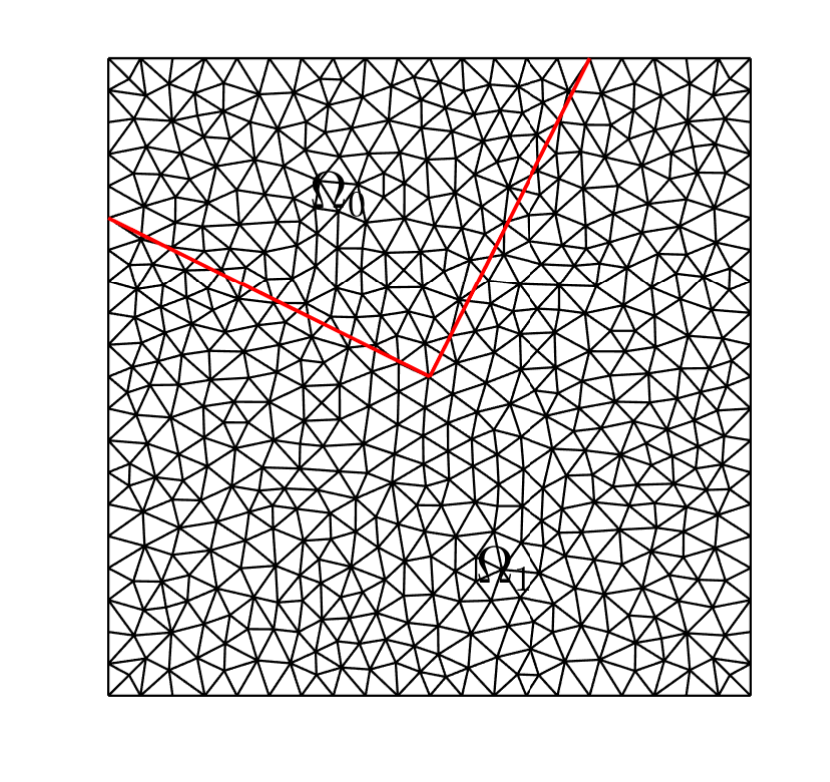}
  \caption{Triangulation for Example 6 with mesh size $h = 1/5$ (left)
    / $h = 1/10$ (right).}
  \label{fig:ex6}
\end{figure}
We note that the interface is only Lipschitz continuous and it has a
kink at $(0, 0)$, see Fig \ref{fig:ex6}. The analytical solution
$u(x, y)$ is given by
\begin{displaymath}
  \begin{aligned}
    u(x, y) &= \begin{cases}
      8, &\quad (x, y) \in \Omega_0, \\
      \sin(x + y), &\quad (x, y) \in \Omega_1 \text{ and } x + y \leq
      0,\\
      x + y, &\quad (x, y) \in \Omega_1 \text{ and } x + y > 0. \\
    \end{cases}\\
  \end{aligned}
\end{displaymath}
We choose $\beta = 1$ over the domain $(-1, 1) \times (-1, 1)$.  The
solution $u(x, y)$ is $C^2$ continuous but not $C^3$ continuous across
the line $x + y = 1$. The numerical errors in terms of $L^2$ norm and
DG energy norm are gathered in Tab \ref{tab:ex6}. It is observed that
when $m = 1, 2$ the numerical solutions converge optimally with rate
$O(h^{m+1})$ for $L^2$ norm and $m$ order for DG energy norm, which
matches with the fact that the exact solution $u$ belongs to
$H^3(\Omega_0 \cup \Omega_1)$.  When $m = 3$ the computed orders of
convergence in $\| \cdot \|_{L^2(\Omega)}$ and $\enorm{\cdot}$ are
about $O(h^{3.5})$ and $O(h^{2.5})$, respectively. A possible
explanation of the convergence orders can be traced to lack of
$H^4$-regularity of the exact solution on the domain $\Omega_1$.

\begin{table}
  \renewcommand\arraystretch{1.2}
  \centering
  \begin{tabular}{p{1.5cm} p{1.5cm} p{2.2cm} p{1.5cm} p{2.2cm}
    p{1.5cm}}
    \hline
    order $m$ &  $h$ & $L^2$ error & order & DG error & order \\
    \hline
    \multirow{5}{*}{$m=1$} & 2.00e-1 & 7.661e-3  & -  & 
    1.835e-1 & -  \\
    & 1.00e-1 & 2.515e-3 & 1.61 & 5.022e-2 & 1.00  \\
     & 5.00e-2 & 6.498e-4 & 1.95 & 2.445e-2 & 1.03 \\
     & 2.50e-2 & 1.653e-4 & 1.97 & 1.199e-2 & 1.02 \\
     & 1.25e-2 & 4.202e-5 & 1.98 & 1.156e-2 & 1.00 \\
    \hline
    \multirow{5}{*}{$m=2$} & 2.00e-1 & 4.727e-4  & -  & 
    9.283e-3 & -  \\
    & 1.00e-1 & 6.423e-5 & 2.85 & 2.393e-3 & 1.95  \\
     & 5.00e-2 & 7.249e-6 & 3.16 & 5.872e-4 & 2.01 \\
     & 2.50e-2 & 9.171e-7 & 2.98 & 1.505e-4 & 1.97 \\
     & 1.25e-2 & 1.126e-7 & 3.02 & 6.401e-5 & 2.02 \\
    \hline
    \multirow{5}{*}{$m=3$} & 2.00e-1 & 1.229e-4  & -  & 
    3.145e-3 & -  \\
    & 1.00e-1 & 1.126e-5 & 3.45 & 3.361e-4 & 2.41  \\
     & 5.00e-2 & 9.603e-7 & 3.55 & 5.721e-5 & 2.53 \\
     & 2.50e-2 & 8.249e-8 & 3.55 & 1.816e-5 & 2.55 \\
     & 1.25e-2 & 6.999e-9 & 3.56 & 3.108e-6 & 2.55 \\
    \hline
  \end{tabular}
  \caption{The convergence orders under $L^2$ norm and DG energy norm
    for Example 6. }
  \label{tab:ex6}
\end{table}

\subsection{3D Example} 
\paragraph{\textbf{Example 7.}} Here we consider a three-dimensional
elliptic interface problem. The domain $\Omega$ is $(0, 1)^3$ and the
spherical interface is given by 
\begin{displaymath}
  \phi(x, y, z) = (x - 0.5)^2 + (y - 0.5)^2 + (z - 0.5)^2 - r^2,
\end{displaymath}
with radius $r = 0.35$. We select $\beta = 1$ in the whole domain and
the exact solution is taken as 
\begin{displaymath}
  u(x, y, z) = \begin{cases}
    \sin(\pi x) \sin(\pi y) \sin(\pi z), \quad &\text{outside
    $\Gamma$}, \\
    e^{x^2 + y^2 + z^2}, \quad & \text{inside $\Gamma$}. \\
  \end{cases}
\end{displaymath}
We adopt a family of tetrahedral meshes with mesh size $h= 1/4$,
$1/8$, $1/16$, $1/32$ to solve the interface problem (see Fig
\ref{fig:3dmesh}). The numerical solutions on the meshes with $h=1/16$
and $h=1/32$ are depicted in Fig \ref{fig:3dsolution} and these two
solutions are obtained with the accuracy $m = 3$. We display the
slices at $y = 0.5$ and $z = 0.5$ of the numerical approximations on
both meshes and both solutions significantly involve a discontinuity
across a spherical, which are accordant with the interface. The
convergence rates under both norms are shown in Fig
\ref{fig:error_ex7}. Clearly, the numerical results are still
consistent with our theoretical predictions. 

\begin{figure}[htp]
  \centering
  \includegraphics[width=1.8in, height=1.8in]{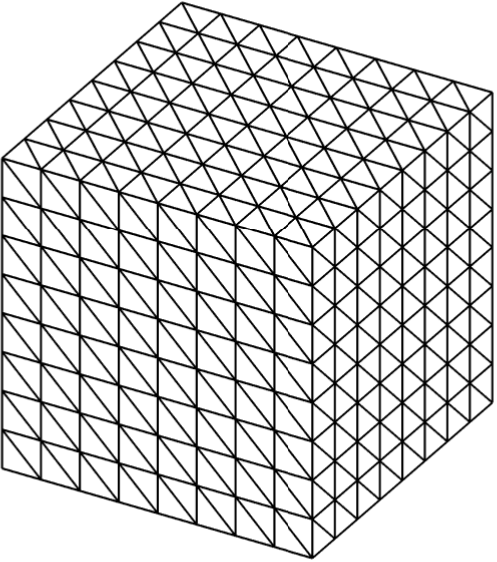}
  \hspace{35pt}
  \includegraphics[width=1.8in, height=1.82in]{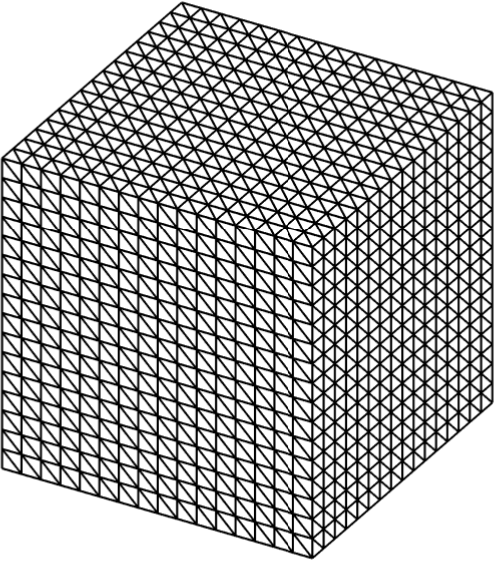}
  \caption{Tetrahedral meshes for example 7 with mesh size $h = 1/8$ (left)
    / $h = 1/16$ (right).}
  \label{fig:3dmesh}
\end{figure}

\begin{figure}[htb]
  \centering
  \includegraphics[width=2.2in]{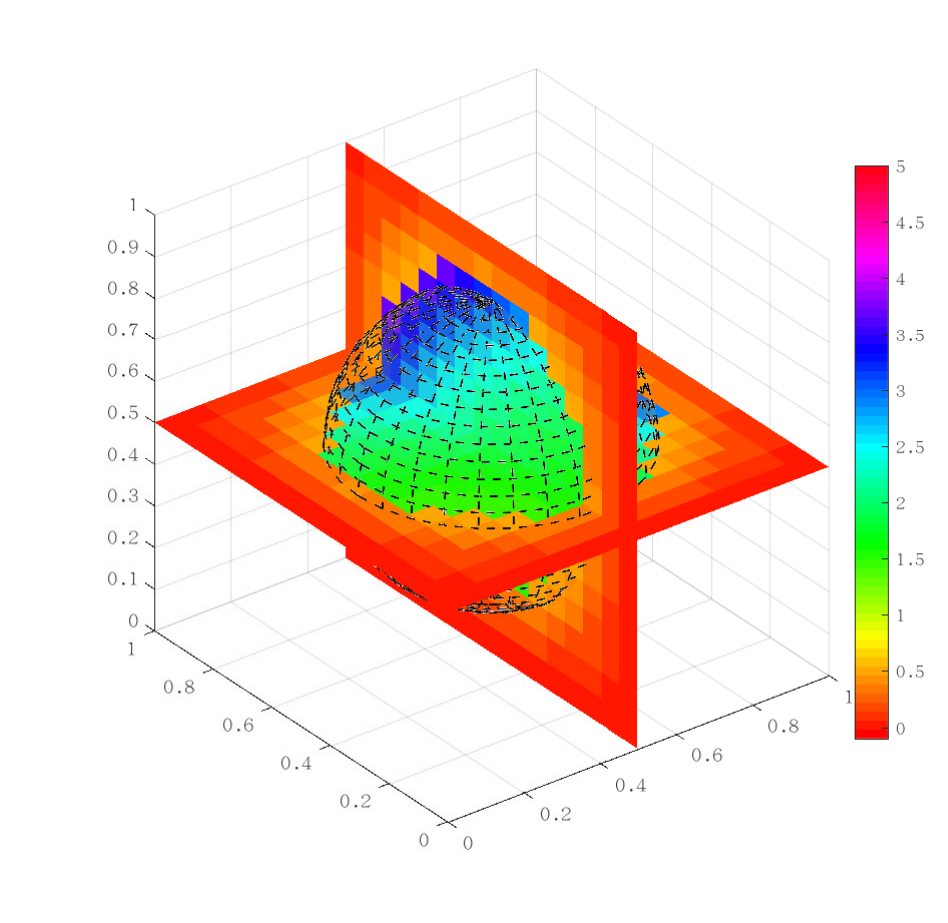}
  \includegraphics[width=2.2in]{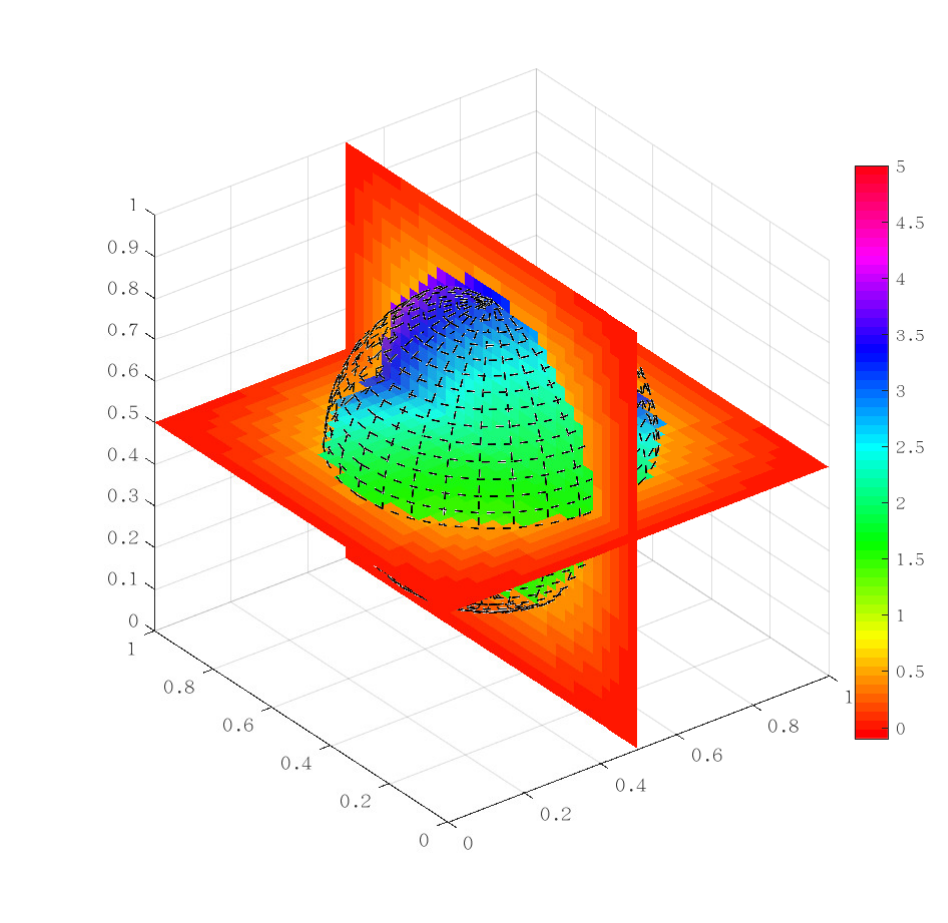}
  \caption{The numerical solution on the tetrahedral mesh with mesh
  size $h=1/16$ (left) / $h=1/32$ (right).}
  \label{fig:3dsolution}
\end{figure}

\begin{figure}[htb]
  \centering
  \includegraphics[width=2.2in]{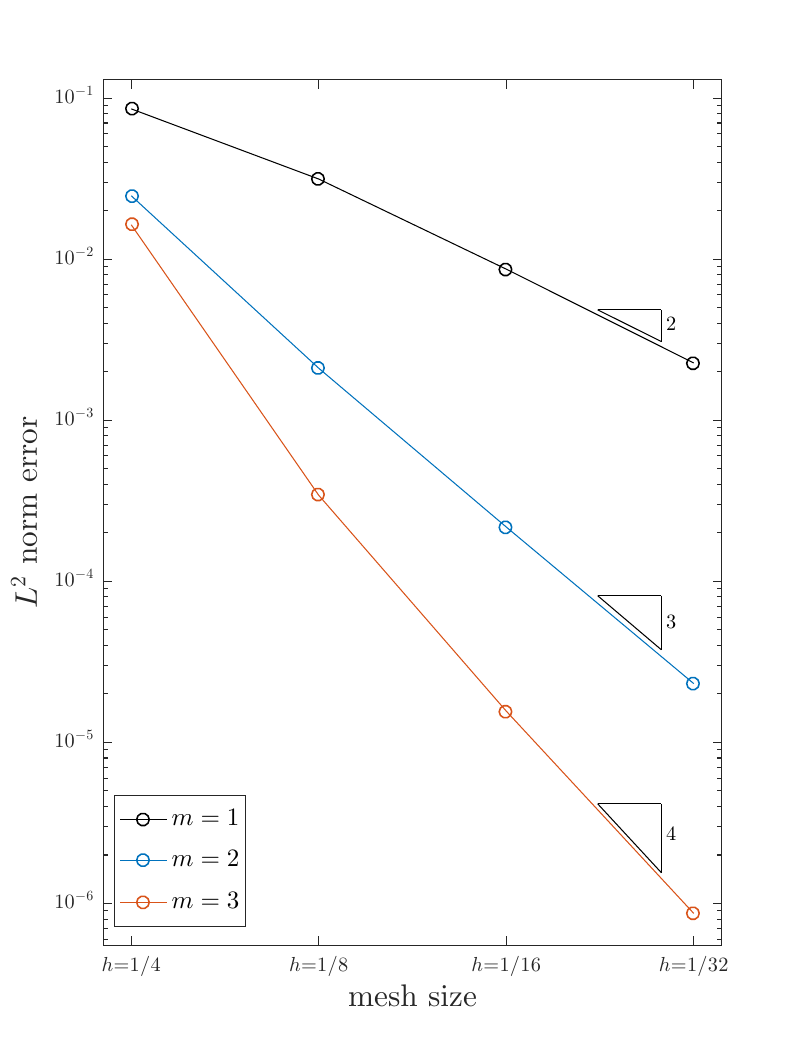}
  \includegraphics[width=2.2in]{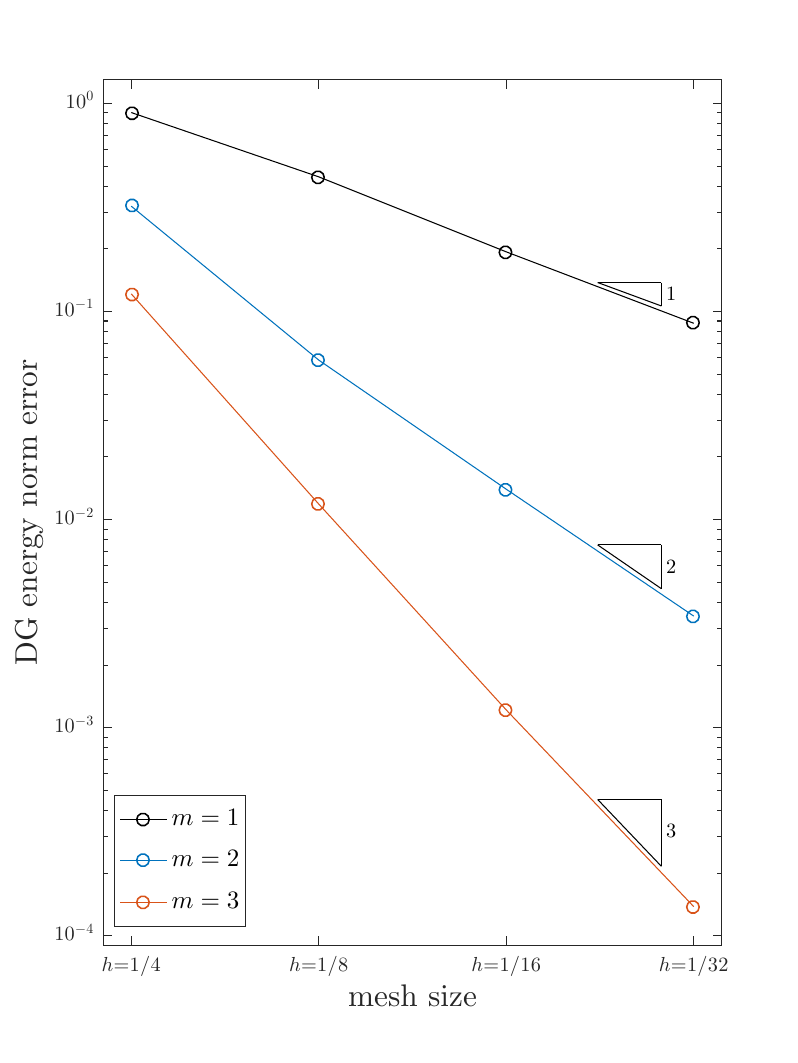}
  \caption{The convergence orders under $L^2$ norm (left) / DG
    energy norm (right) for Example 7.}
  \label{fig:error_ex7}
\end{figure}

\paragraph{\textbf{Example 8.}} 
In this example, we consider a three-dimensional elliptic problem
\cite{Wei2018spatially} with a smooth interface that is governed by
the following level set function (see Fig \ref{fig:ex8interface}), 
\begin{displaymath}
  \phi(x, y, z) = \left( (2.5(x - 0.5))^2 + (4.2(y - 0.5))^2 + (2.5(z
  - 0.5))^2 + 0.9\right)^2 - 64(y - 0.5)^2 - 1.3.
\end{displaymath}
The exact solution is choose to be 
\begin{displaymath}
  u(x, y, z) = \begin{cases}
    \cos(\pi x) \cos(\pi y) \cos(\pi z), \quad &\text{outside
    $\Gamma$}, \\
    5e^{x^2 + y^2 + z^2}, \quad & \text{inside $\Gamma$}. \\
  \end{cases}
\end{displaymath}
The domain $\Omega$ is taken to be $(0, 1)^3$ and the coefficient
$\beta$ is fixed as $1$. We also use the tetrahedral meshes with mesh
size $h=1/4$, $1/8$, $1/16$, $1/32$ for solving the problem (see Fig
\ref{fig:3dmesh}). The slices of the numerical solution with the
accuracy $m=3$ on the tetrahedral mesh with $h = 1/32$ at $x = 0.5$
and at $z = 0.5$ are depicted in Fig \ref{fig:ex8slice}. It is clear
that the discontinuity of the numerical solution sketches a curve
which matches with the interface given by the level set function (see
\ref{fig:ex8interface}). We also display the convergence history of
the numerical approximation under both $L^2$ norm and DG energy norm
in Fig \ref{fig:error_ex8}. The convergence rate of $L^2$ error may
seem less than the predicted value when the accuracy $m=1$. The rate 
is gradually more close to the theoretical value and we may expect the
rate would go back to $O(h^2)$ as the mesh size tends to zero. For
$m=2$ and $m=3$, the computed convergence rates under both error
measurements are in agreement with the theoretical results.  

\begin{figure}[htp]
  \centering
  \includegraphics[width=2in]{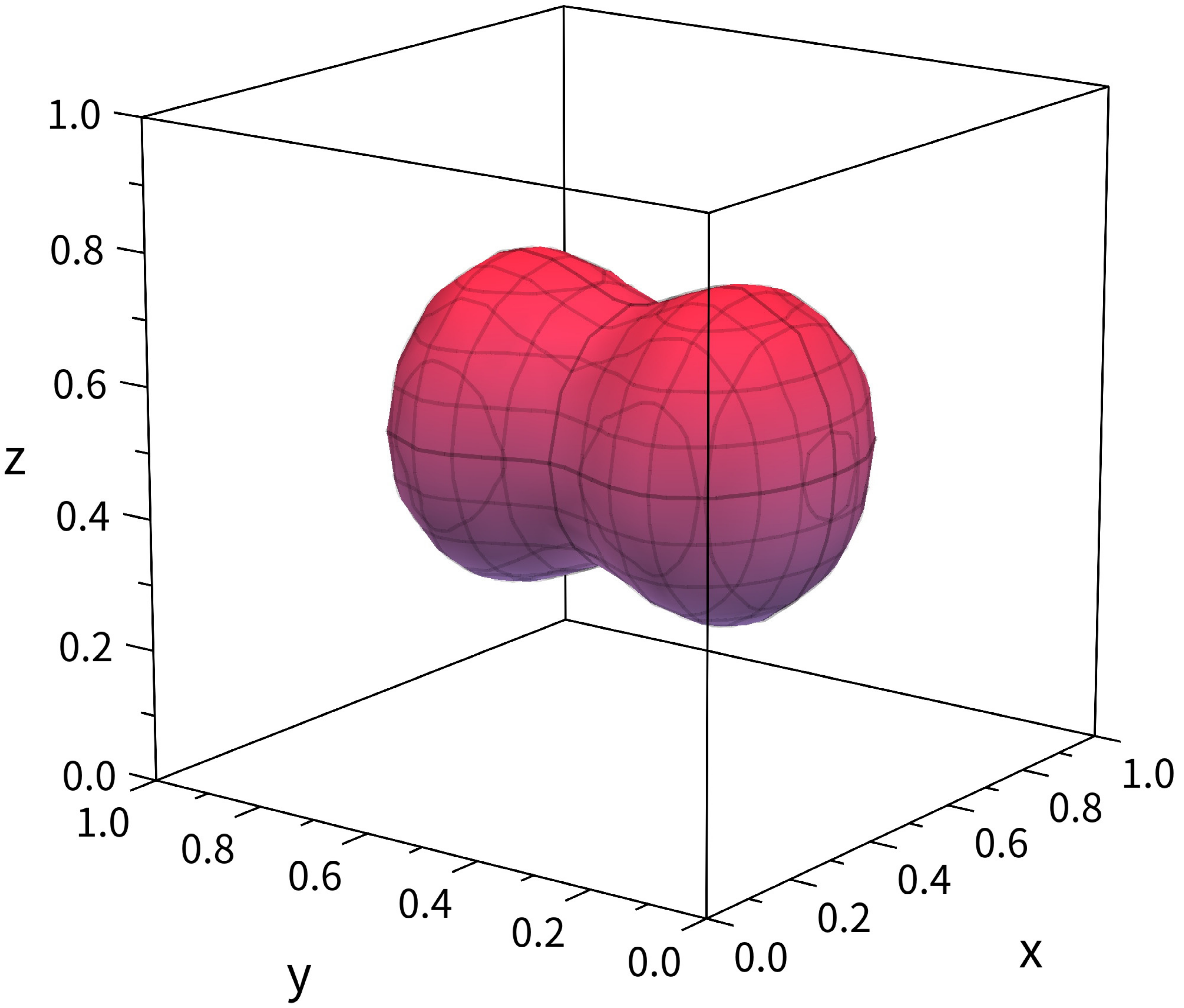}
  \hspace{30pt}
  \includegraphics[width=1.7in]{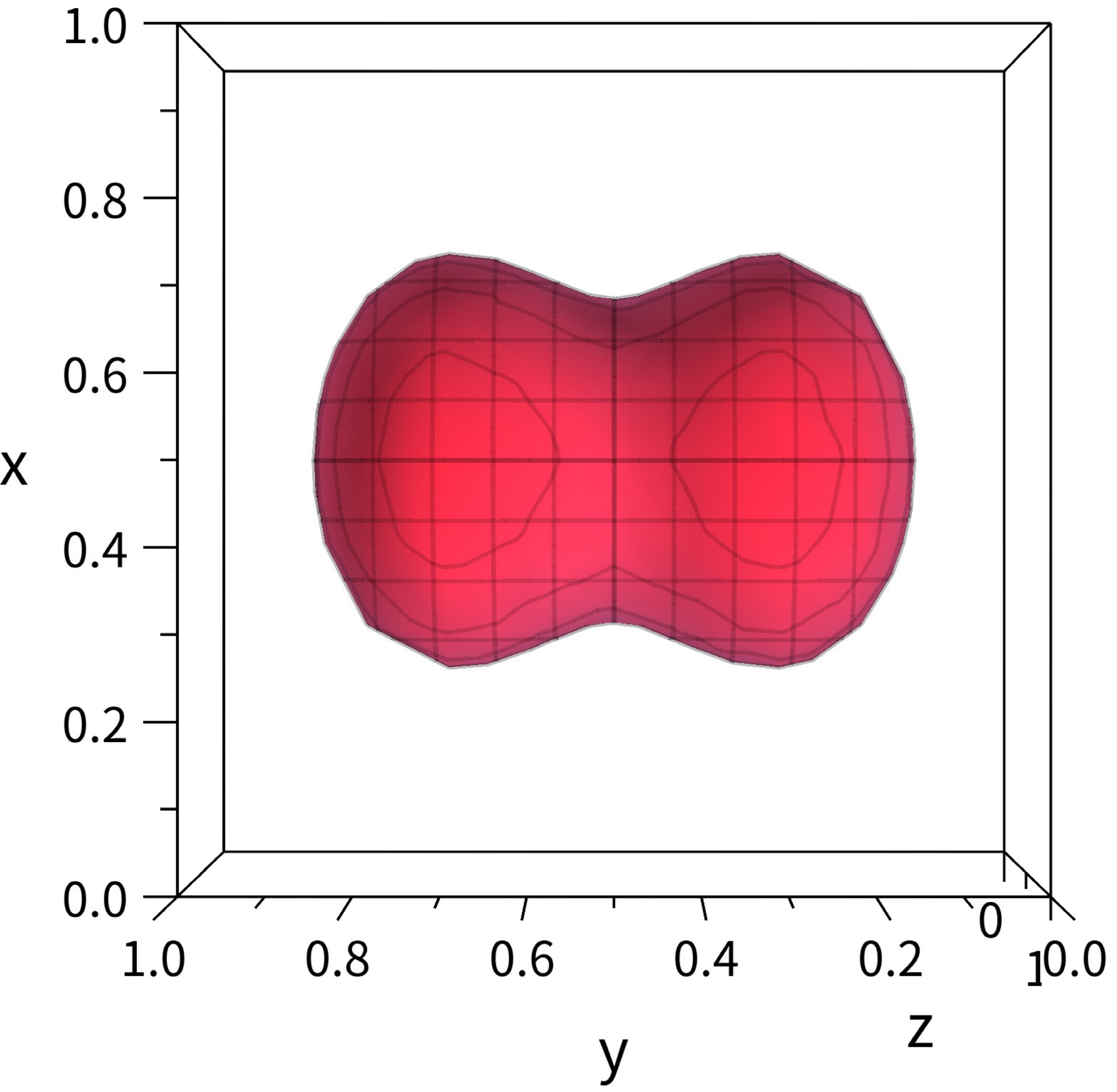}
  \caption{The interface of Example 8.}
  \label{fig:ex8interface}
\end{figure}

\begin{figure}[htp]
  \centering
  \includegraphics[width=2.2in]{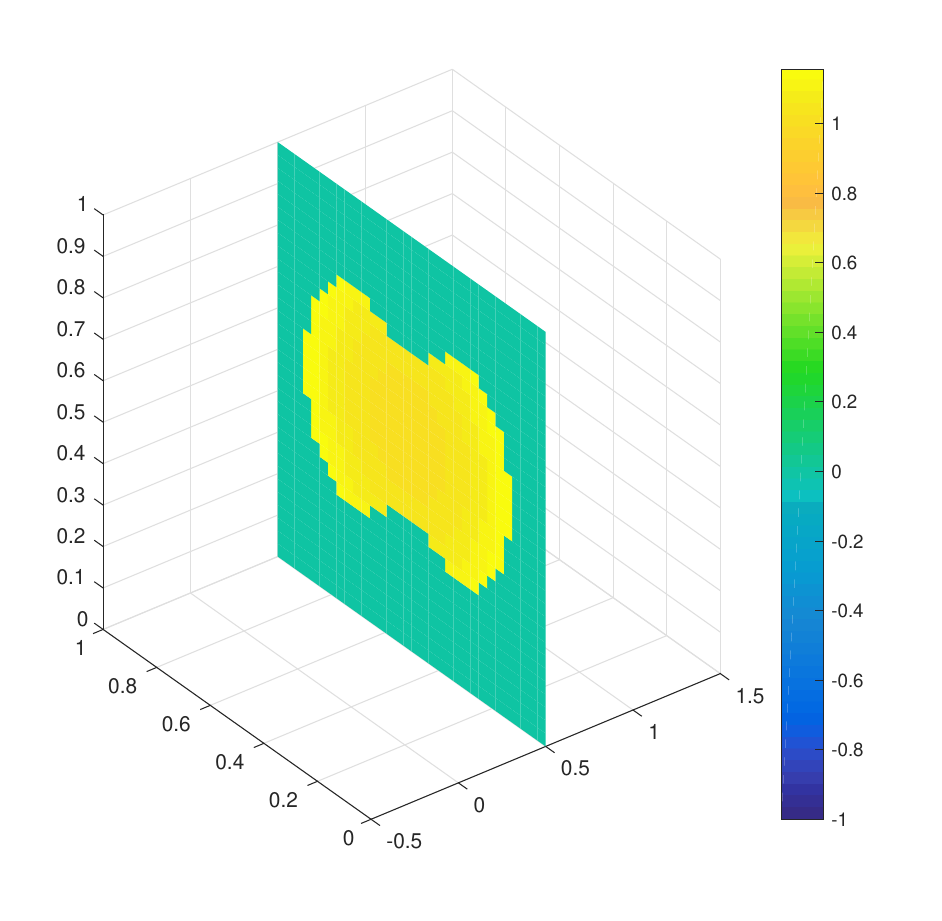}
  \hspace{30pt}
  \includegraphics[width=2.2in]{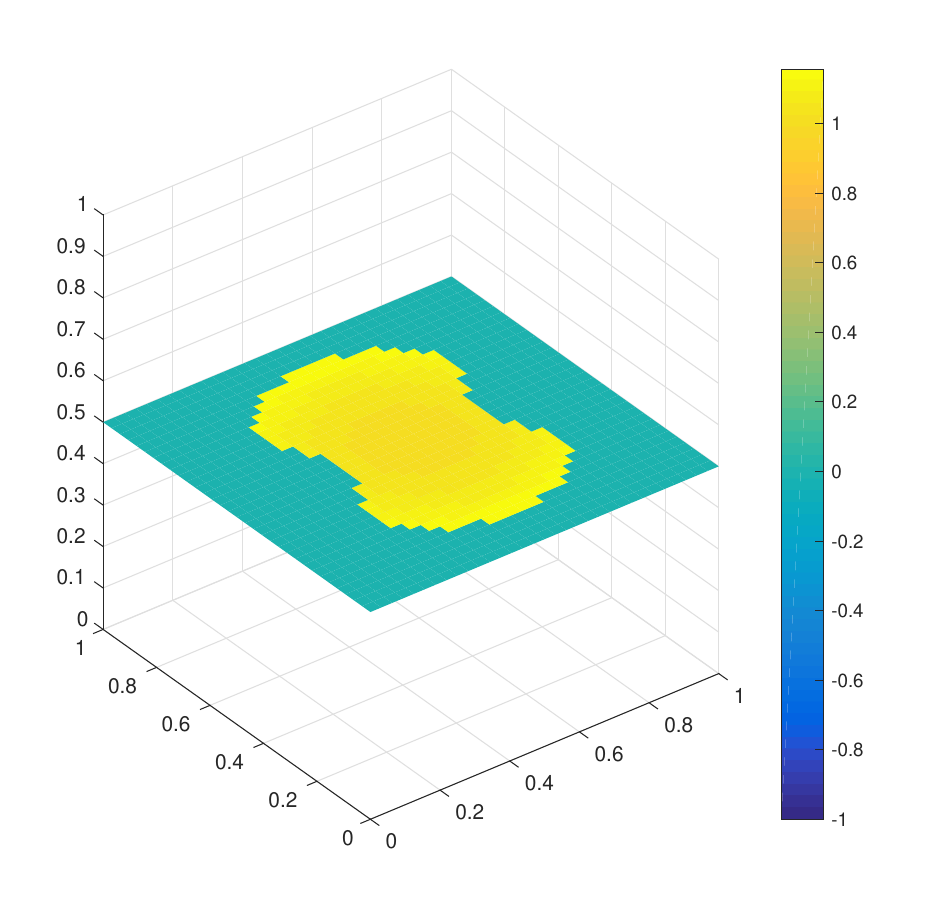}
  \caption{The slice of the numerical solution at $x = 0.5$ (left) /
  at $z = 0.5$ (right).}
  \label{fig:ex8slice}
\end{figure}

\begin{figure}[htb]
  \centering
  \includegraphics[width=2.2in]{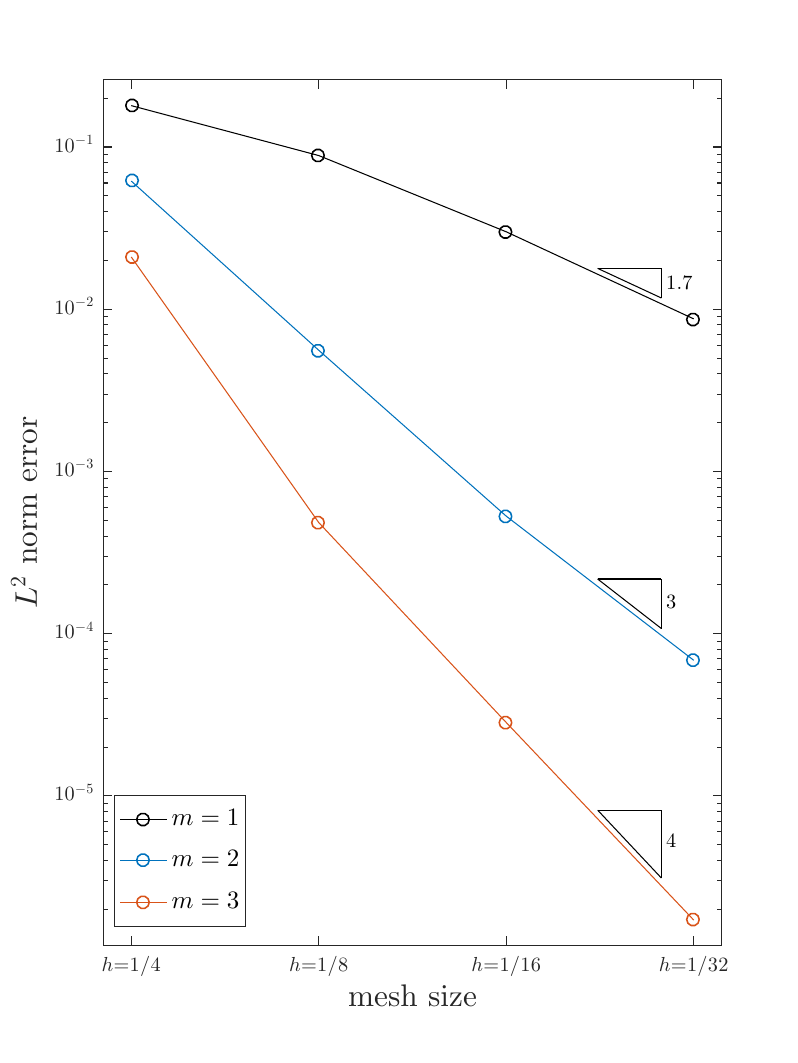}
  \includegraphics[width=2.2in]{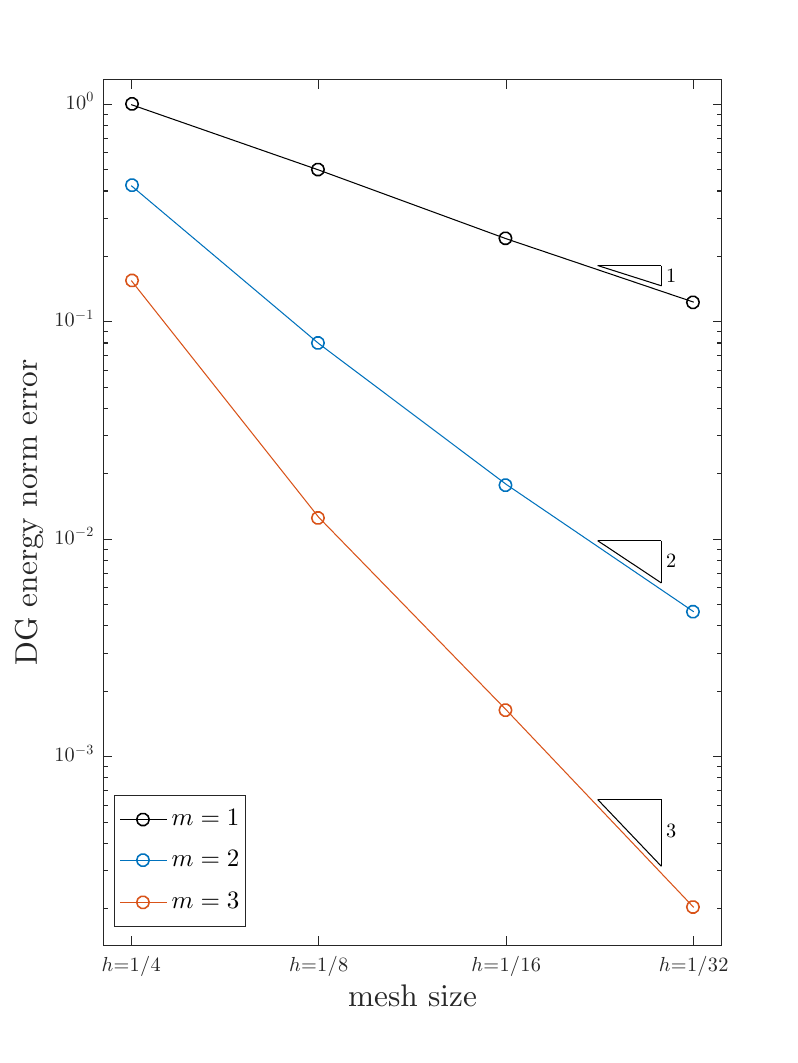}
  \caption{The convergence orders under $L^2$ norm (left) / DG
    energy norm (right) for Example 8.}
  \label{fig:error_ex8}
\end{figure}

\subsection{Integrals on Cut Element} 
In our method, computing the following types of integrals defined on
the cut element is an important issue, 
\begin{displaymath}
  \int_{K^0} v(\bm{x}) \d{x}, \quad \int_{K^1}
  v(\bm{x}) \d{x}, \quad \int_{\Gamma_K} v(\bm{x}) \d{s},
\end{displaymath}
where $K \in \MThG$ is a cut element and $K^0 = K \cap \Omega_0$, $K^1
= K \cap \Omega_1$ and $\Gamma_K = K \cap \Gamma$. Here we list two
numerical methods for computing these integrals. The first is we
generate highly accurate quadrature points and weights corresponding
to the domain $K^0$, $K^1$ and the interface $\Gamma_K$. We refer to
\cite{Cui2019quadratures, Muller2013highly, Huang2017unfitted} for
some approaches about finding such quadrature points and weights. The
computational cost of the first method is much more expensive than
ordinary numerical quadrature methods.  The second one is we
approximate the interface $\Gamma_K$ by planes or lines inside the
element $K$, see Fig ~\ref{fig:plane} for an example. In this case, we
only need to generate quadrature points and weights for polygons or
polyhedrons. The computational cost is much less than the first method
but the result is less accurate. We refer to \cite{Wang2017improved}
for more details about this method. 

Here we make a comparison between two methods. We solve the Example 7
by both two numerical quadrature methods. We call the C subroutines in
PHG package \cite{Xu2016parallel, Cui2019quadratures} to generate
highly accurate quadrature points and weights for the cut
tetrahedrons. For the second methods and for element $K \in \MThG$, we
let $K^0$ be approximated by $\wt{K}^0$ and let $K^1$ be approximated
by $\wt{K}^1$. The actual computational domains $\wt{\Omega}_0$ and
$\wt{\Omega}_1$ are then given as
\begin{displaymath}
  \wt{\Omega}_i = \left( \bigcup_{K \in \MThG} \wt{K}^i \right)
  \bigcup \left( \bigcup_{K \in \MTh^i \backslash \MThG} K \right),
  \quad i = 0, 1.
\end{displaymath}
We list the $L^2$ errors $\|u - u_h^1\|_{L^2(\Omega_0 \cup \Omega_1)}$
and $\|u - u_h^2\|_{L^2(\wt{\Omega}_0 \cup \wt{\Omega}_1)}$ in Tab
\ref{tab:errorcompare}, where $u_h^1$ and $u_h^2$ are the numerical
solutions obtained by the first and second numerical quadrature
methods, respectively. From Tab \ref{tab:errorcompare}, we observe
that the two errors are gradually closer to each other when the mesh
size tends to zero. We note that both quadrature methods work in our
numerical scheme \substitute{the second one is more accurate but extra
computational cost is required.}{\revise{and the first one is more
accurate but much more computational cost is required.}}

\begin{table}
  \centering
  \renewcommand\arraystretch{1.5}
  \begin{tabular}{p{1cm} |p{3cm} | p{1.5cm} |  p{1.5cm} |  p{1.5cm} |p{1.5cm} }
    \hline\hline
    & & $h=1/4$ & $h=1/8$ & $h=1/16$ & $h=1/32 $ \\
    \hline
    \multirow{2}{*}{$m=1$} & $\|u - u_h^1\|_{L^2(\Omega_0 \cup
    \Omega_1)}$
    & 7.1514e-2  & 3.0977e-2   & 1.0696e-2 & 2.8099e-3 \\
    \cline{2-6}
    & $\|u - u_h^2\|_{L^2(\wt{\Omega}_0 \cup \wt{\Omega}_1)}$ 
    & 1.0216e-1 &  3.3384e-2 & 1.1173e-2 & 2.8274e-3 \\
    \hline
    \multirow{2}{*}{$m=2$} & $\|u - u_h^1\|_{L^2(\Omega_0 \cup
    \Omega_1)}$ 
    & 6.1523e-2  & 2.0271e-3   & 1.8898e-4 & 2.3557e-5 \\
    \cline{2-6}
    & $\|u - u_h^2\|_{L^2(\wt{\Omega}_0 \cup \wt{\Omega}_1)}$ 
    & 3.4460e-2  & 2.4149e-3   & 2.1523e-4 & 2.3799e-5 \\
    \hline
    \multirow{2}{*}{$m=3$} & $\|u - u_h^1\|_{L^2(\Omega_0 \cup
    \Omega_1)}$ 
    & 5.1865e-2 & 3.9632e-4  & 1.7698e-5 & 9.0215e-7  \\
    \cline{2-6}
    & $\|u - u_h^2\|_{L^2(\wt{\Omega}_0 \cup \wt{\Omega}_1)}$ 
    & 1.7197e-2 & 3.7255e-4 & 1.5553e-5 & 8.7805e-7 
    \\
    \hline\hline
  \end{tabular}
  \caption{The $L^2$ errors  $\|u - u_h^1\|_{L^2(\Omega_0 \cup
  \Omega_1)}$ and $\|u - u_h^2\|_{L^2(\wt{\Omega}_0 \cup
  \wt{\Omega}_1)}$.  }
  \label{tab:errorcompare}
\end{table}

\begin{figure}[htp]
  \centering
  \begin{tikzpicture}
    \draw[thick] (-0.9, 1.5) --  (0, 0.5);
    \draw[thick] (0.0, 0.5)  -- (1.1, 1.5);
    \draw[thick, dashed] (-0.9, 1.5) -- (1.1, 1.5);
    \draw[fill, gray] (-0.9, 1.5) -- (0.0, 0.5) -- (1.1, 1.5) -- (-0.9,
    1.5);
    \draw[thick, dashed, black] (-2, 0) -- (2, 0);
    \draw[thick, black] (-2, 0) -- (0.2, 3.0) -- (2, 0);
    \draw[thick, black] (-2, 0) -- (-0.2, -2.0) -- (2, 0);
    \draw[thick, black] (0.2, 3.0) -- (-0.2, -2.0);
    \draw[thick, black] (-0.9, 1.5) to [in = 100, out = -10] (0, 0.5);
    \draw[thick, black] (0.0, 0.5) to [in = 200, out = 70] (1.1, 1.5);
    \draw[thick, black, dashed] (-0.9, 1.5) to [in = 150, out = 30] (1.1, 1.5);
    \node at (0.45, 2) {$\Gamma_K$};
    \node at (1, 0.8) {plane};
    \draw[thick, black] (4, -1.5) -- (8, -1.5) -- (6, 2) -- (4,
    -1.5);
    \draw[thick, black] (5, 0.25) to [out = 20, in = 160] (7, 0.25);
    \draw[thick, gray] (5, 0.25) -- (7, 0.25);
    \node[below] at (6, 0.25) {line};
  \end{tikzpicture}
  \caption{The interface inside a cut element $K$ for $d = 3$ (left) /
  $d = 2$ (right).}
  \label{fig:plane}
\end{figure}
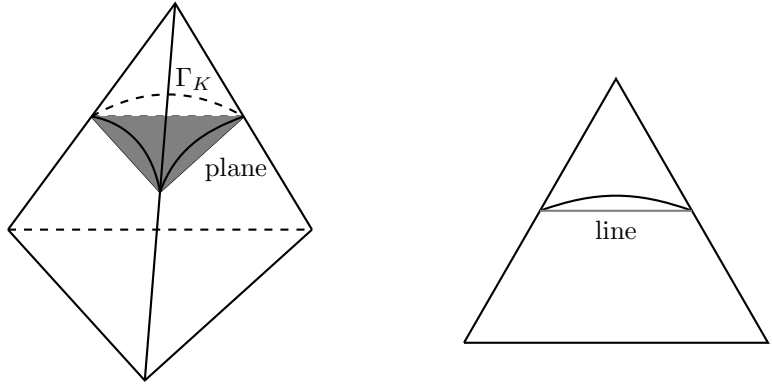

\subsection{Efficiency comparison}
Hughes et al. \cite{hughes2000comparison} point out that the number of
unknowns of a discretized problem is a proper indicator for the
efficiency of a numerical method. To show the efficiency in DOFs of
our method, we make a comparison among the unfitted DG method
\cite{Massjung2012unfitted}, the unfitted penalty finite element
method \cite{Wu2012unfitted, Xiao2020high} and our method by
solving the two-dimensional elliptic interface problem. The first
method adopts the standard discontinuous finite element space, and the
second method employs the traditional continuous finite element space.
The solution and the partition are taken from Example 1. In Fig
\ref{fig:compare}, we plot the $L^2$ norm of the error of three
methods against the number of degrees of freedom with $1 \leq m \leq
3$.

One see that for the low orders of approximation($m = 1$), the penalty
FE method is the most efficient method. For $m=2$, our method shows
almost the same efficiency as the penalty FE method. For the high
order accuracy($m=3$), our method performs better than the other
methods. 
\begin{figure}[htb]
  \centering
  \includegraphics[width=.32\textwidth]{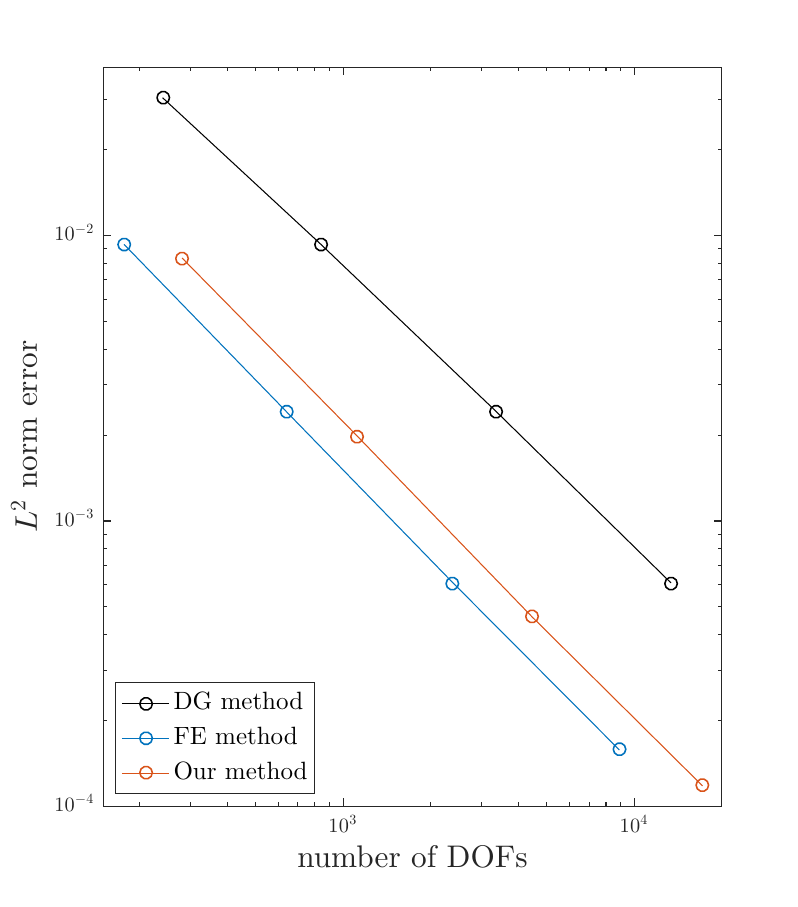}
  \includegraphics[width=.32\textwidth]{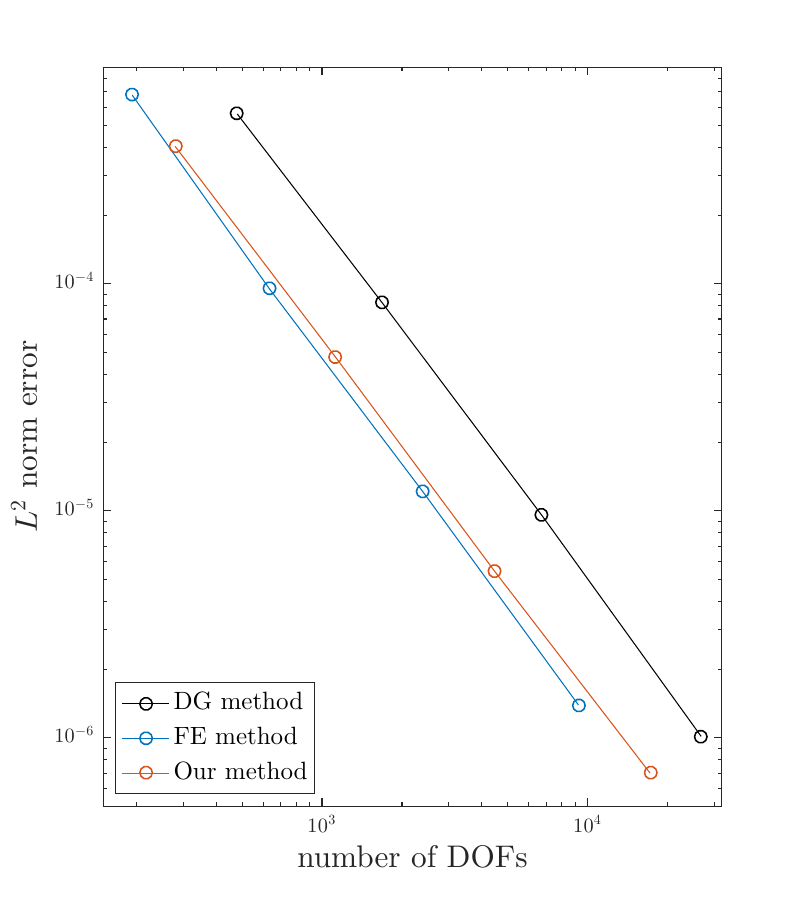}
  \includegraphics[width=.32\textwidth]{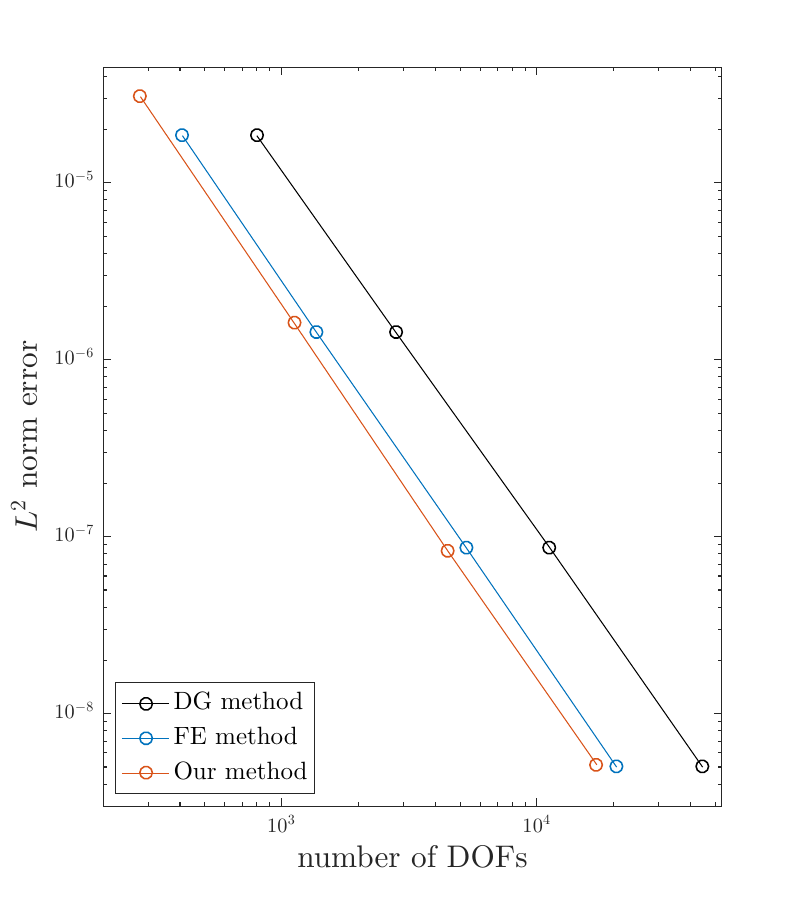}
  \caption{Comparison of the $L^2$ errors in number of DOFs by three
    methods with $m=1$, $2$, and $3$.}
  \label{fig:compare}
\end{figure}


\section{Conclusion}
\label{sec:conclusion}
We proposed a new discontinuous Galerkin method for elliptic interface
problem. The approximation space is constructed by solving the local
least squares problem. We proved optimal convergence orders in both
$L^2$ norm and DG energy norm.  A series of numerical results confirm
our theoretical results and exhibit the flexibility, robustness and
efficiency of the proposed method.

\section*{Acknowledgements}
The authors would like to thank the anonymous referees sincerely for
their constructive comments that improve the quality of this paper.
This research was supported by the Science Challenge Project (No.
TZ2016002) and the National Science Foundation in China (No.
11971041).


\begin{appendix}
  \section{Construction of Element Patch}
  \label{sec:cep}
  \revise{Here we present the algorithm to the construction of the
  element patch in Alg ~\ref{alg:patch} and} we also give some an
  example of constructing element patches. We consider a circular
  interface. Let $\Omega_1$ be the domain inside the circle and
  $\Omega_0 = \Omega \backslash \Omega_1$. For element $K \in \MTh^0
  \backslash \MThG$, the construction of $S^0(K)$ is presented in Fig
  \ref{fig:build_patch0}.
  \begin{figure}
    \centering
    \includegraphics[width=1.00\textwidth]{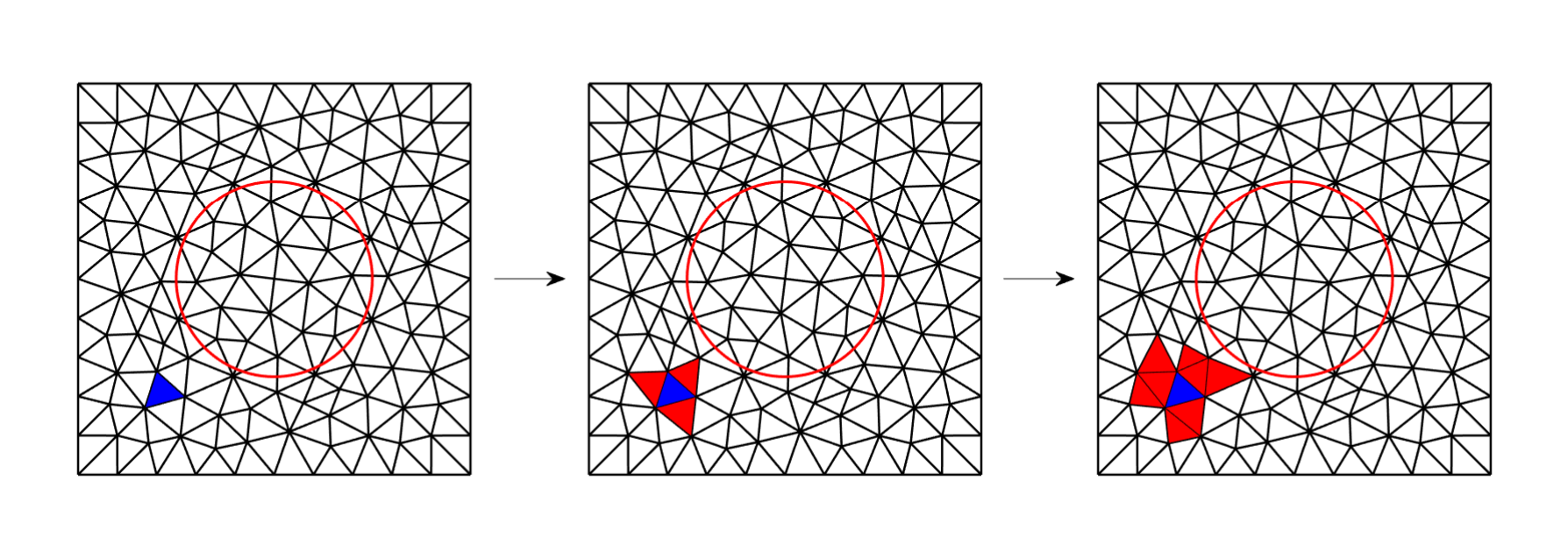}.
    \vspace{-30pt}
    \caption{Example to build element patch $S^0(K)$ for $K \in \MTh^0
    \backslash \MThG$.}
    \label{fig:build_patch1}
  \end{figure}
  For element $K \in \MTh^1 \backslash \MThG$, the construction of
  $S^1(K)$ is presented in Fig \ref{fig:build_patch0}.
  \begin{figure}
    \centering
    \includegraphics[width=1.00\textwidth]{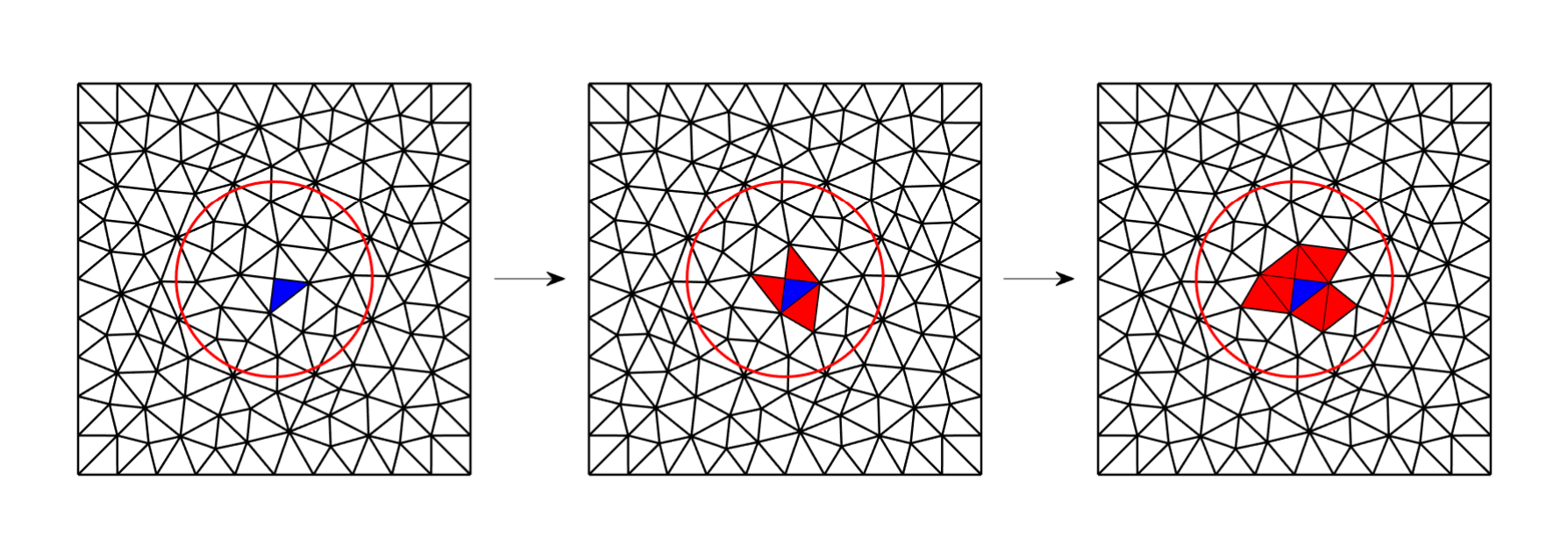}.
    \vspace{-30pt}
    \caption{Example to build element patch $S^1(K)$ for $K \in \MTh^1
    \backslash \MThG $.}
    \label{fig:build_patch0}
  \end{figure}

  \begin{algorithm}[htb]
    \caption{Construction of Element Patch}
    \label{alg:patch}
    \begin{algorithmic}[1]
    \renewcommand{\algorithmicrequire}{\textbf{Input:}}
    \REQUIRE
    partition $\MTh$ and a uniform threshold $\#
    S(K)$; \\
    \renewcommand{\algorithmicrequire}{\textbf{Output:}}
    \REQUIRE
    the element patches $S^0(K)$ for all $K$ in $\MTh^0$ and the
    element patches $S^1(K)$ for all $K$ in $\MTh^1$;
    \FOR{$i \in \{0, 1\}$}
    \FOR{each $K \in \MTh^i \backslash \MThG$}
    \STATE{set $t=0$, $S_t^i(K) = \left\{ K \right\}$;}
    \WHILE{the cardinality of $S_t^i(K)$ $<$ $\# S(K)$}
    \STATE{initialize the set $S_{t+1}^i(K) = S_t^i(K)$;}
    \FOR{each $K \in S_t^i(K)$}
    \STATE{let $N(K)$ be the face-neighbouring elements of $K$;}
    \FOR{each $\wt{K} \in N(K)$}
    \IF{$\wt{K} \notin S_{t+1}^i(K)$ and $\wt{K} \in \MTh^i$}
    \STATE{add $\wt{K}$ to $S_{t+1}^i(K)$;}
    \ENDIF
    \IF{the cardinality of $S_t^i(K)$ = $\# S(K)$}
    \STATE{break while;}
    \ENDIF
    \ENDFOR
    \ENDFOR
    \STATE{let $t = t+1$;}
    \ENDWHILE
    \STATE{let $S^i(K) = S_t^i(K)$;}
    \ENDFOR
    \ENDFOR
    \FOR{each $K \in \MThG$}
    \STATE{seek $K_\circ^0$ and $K_\circ^1$ and let $S^0(K) =
    S^0(K_\circ^0)$ and $S^1(K) = S^1(K_\circ^1)$;}
    \ENDFOR
    \end{algorithmic}
  \end{algorithm}

  \section{1D Example}
  \label{sec:1dexample}
  Here we present a one-dimensional example to illustrate our method.
  We consider the interval $\Omega = [-1, 1]$ which is divided into two
  parts $\Omega_0 = (-1, -0.2) $ and $\Omega_1 = (-0.2, 1)$. We
  partition $\Omega$ into 8 elements $\left\{ K_1, K_2, \cdots, K_8
  \right\}$ with uniform spacing.
  \begin{figure} 
    \centering
    \hspace{-25pt}
    \begin{tikzpicture}[xscale=5.5, yscale=5]
      \draw[thick] (-1, 0) -- (1,0);
      \foreach \x in{0,...,8}
      \draw[thick] (-1 + \x*0.25, -0.002) -- (-1 + \x*0.25, 0.02);
      \foreach \x in{0,...,7}
      \draw[fill=black] (-0.875 + \x*0.25, 0) circle [radius=0.01];
      \foreach \x in{1,...,8}
      \node[below] at (-1.08+ \x*0.25, 0) {\footnotesize $K_{\x}$};
      \foreach \x in{1,...,8}
      \node[above right] at (-1.135+ \x*0.25, 0) {\footnotesize $x_{\x}$};
      \draw[thick, dashed, <->] (-1, -0.15) -- (-0.2, -0.15);
      \draw[thick, dashed] (-1, -0.13) -- (-1, 0);
      \draw[thick, dashed] (-0.2, -0.13) -- (-0.2, 0.02);
      \node[below] at (-0.6, -0.15) {$\Omega_0$};
      \draw[thick, dashed, <->] (1, -0.15) -- (-0.2, -0.15);
      \draw[thick, dashed] (1, -0.13) -- (1, 0);
      \node[below] at (0.399, -0.15) {$\Omega_1$};
    \end{tikzpicture}
    \vspace{-15pt}
    \caption{The uniform grid on $[-1, 1]$.}
    \label{fig:uniformgrid}
  \end{figure}
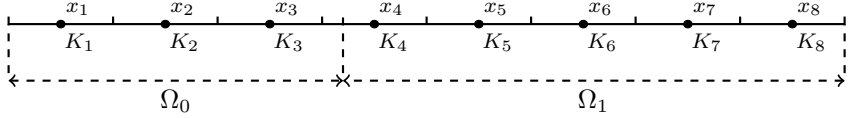
  $\left\{ x_1, x_2, \cdots, x_8 \right\}$ are the set of collocations
  where $x_i$ is the midpoint of the element $K_i$. Since $\MThG =
  \left\{ K_4 \right\}$, we construct element patches for elements in
  $\MThB$. The element patches could be constructed as 
  \begin{displaymath}
    S^0(K_1) = \left\{ K_1, K_2 \right\}, \quad S^0(K_2) = \left\{
    K_2, K_3 \right\}, \quad S^0(K_3) = \left\{ K_2, K_3, K_4
    \right\}, 
  \end{displaymath}
  \begin{displaymath}
     \begin{aligned}
       S^1(K_5) &= \left\{ K_4, K_5, K_6 \right\}, \quad S^1(K_6) =
       \left\{ K_5, K_6, K_7 \right\}, \\
       S^1(K_7) &= \left\{ K_6, K_7 \right\}, \quad S^1(K_8) = \left\{
       K_7, K_8 \right\}.
     \end{aligned}
  \end{displaymath}
  Then for element $K_4$ it is clear that $K_\circ^0 = K_3$ and
  $K_\circ^1 = K_5$, and the element patches of $K_4$ are 
  \begin{displaymath}
    \begin{aligned}
    S^0(K_4) &= S^0(K_3) = \left\{K_2, K_3, K_4 \right\},  \\
    S^1(K_4) &= S^1(K_5) = \left\{ K_4, K_5, K_6 \right\}.
  \end{aligned}
  \end{displaymath}
  Then we would solve the least squares problem on every patch. We
  take $S^0(K_3)$ for an example, for a continuous function $g$ and
  $m=1$ the least squares problem is written as
  \begin{displaymath}
    \mathop{\arg \min}_{(a, b) \in \mb R} \sum_{i = 2}^4 |(ax_i + b)
    - g(x_i)|^2.
  \end{displaymath}
  It is easy to get the unique solution
  \begin{displaymath}
    (a,b)^T = (A^TA)^{-1}A^Tq, 
  \end{displaymath}
  where
  \begin{displaymath}
    A = \begin{bmatrix}
      1 & x_2 \\ 1 & x_3 \\ 1 & x_4 \\
    \end{bmatrix}, \qquad q = \begin{bmatrix}
      g(x_2) \\ g(x_3) \\ g(x_4) \\
    \end{bmatrix}.
  \end{displaymath}
  We note that the matrix $(A^TA)^{-1}A^T$ has no relationship to the
  function $g$ and contains all information of all $\lambda_K^i$ on
  element $K_2$. Hence we store the matrix $(A^TA)^{-1}A^T$ for every
  element patch to represent all $\lambda_K^i$. It is in the same way
  when we deal with the high dimensional problem.
\end{appendix}


\bibliographystyle{amsplain}
\bibliography{../ref.bib}
\end{document}